\documentclass[reqno,letterpaper,11pt]{article}

\usepackage{hyperref}
\usepackage{amsthm}
\usepackage{times}
\usepackage{amsmath}
\usepackage{amsfonts,amssymb,amsmath,latexsym,wasysym,mathrsfs,bbm,stmaryrd,esint}
\usepackage[all]{xy}
\usepackage[usenames]{color}
\usepackage{epsfig}
\usepackage[mathcal]{euscript}
\usepackage{graphics,moreverb,hangcaption}

\newtheorem{theorem}{Theorem}[section]
\newtheorem{proposition}[theorem]{Proposition}
\newtheorem{corollary}[theorem]{Corollary}
\newtheorem{lemma}[theorem]{Lemma}
\theoremstyle{definition}
\newtheorem{notation}[theorem]{Notation}
\newtheorem{definition}[theorem]{Definition}
\newtheorem{remark}[theorem]{Remark}
\newtheorem{example}[theorem]{Example}

\numberwithin{equation}{section}

\long\def\symbolfootnote[#1]#2{\begingroup%
\def\thefootnote{\fnsymbol{footnote}}\footnote[#1]{#2}\endgroup}

\newcommand{\C}{\mathbb{C}}
\newcommand{\RR}{\mathbb{R}}
\newcommand{\R}{\mathbb{R}}

\newcommand{\e}{\epsilon}

\newcommand{\LSH}{\mathrm{LSH}}
\newcommand{\supp}{\mathrm{supp}}
\newcommand{\del}{\partial}

\begin{document}

\title{Strong Logarithmic Sobolev Inequalities for Log-Subharmonic Functions}
\author{Piotr Graczyk\thanks{Partly supported by ANR-09-BLAN-0084-01} \\
  Universit\'e d'Angers, 2 Boulevard Lavoisier \\ 49045 Angers Cedex 01, France \\
  \texttt{Piotr.Graczyk@univ-angers.fr}
  \and 
  Todd Kemp\thanks{Partly supported by NSF Grant DMS-1001894 and NSF CAREER Award DMS-1254807} \\
  Department of Mathematics, University of California San Diego \\
  9500 Gilman Drive, La Jolla, CA \, 92093-0112 \, USA \\
  \texttt{tkemp@math.ucsd.edu}
  \and
  Jean-Jacques Loeb \\
  Universit\'e d'Angers, 2 Boulevard Lavoisier \\ 49045 Angers Cedex 01, France \\
  \texttt{Jean-Jacques.Loeb@univ-angers.fr}
}

\date{\today}

\maketitle

\begin{abstract} We prove an intrinsic equivalence between strong hypercontractivity (\ref{sHC}) and a strong logarithmic Sobolev inequality (\ref{sLSI}) for the cone of logarithmically subharmonic ($\LSH$) functions.  We introduce a new large class of measures, Euclidean regular and exponential type, in addition to all compactly-supported measures, for which this equivalence holds.  We prove a Sobolev density theorem through $\LSH$ functions, and use it to prove the equivalence of (\ref{sHC}) and (\ref{sLSI}) for such log-subharmonic functions.
\end{abstract}

\tableofcontents

\section{Introduction} \label{sect Introduction}

%{\bf ADD EARLY CITATION TO \cite{Gross Grothaus}}

In this paper we study strong versions of logarithmic Sobolev inequalities (\ref{sLSI})  and strong hypercontractivity (\ref{sHC}) in the real spaces $\RR^n$ and for logarithmically subharmonic ($\LSH$) functions, continuing our research published in  \cite{GKL} and solving the conjecture on the equivalence between (\ref{sHC}) and (\ref{sLSI}) formulated in \cite[Remark 5.11]{GKL}. The main difficulty to overcome,
as already noticed by Gross and Grothaus in \cite{Gross Grothaus}, was efficient approximating of (logarithmically) subharmonic functions.  %%  We refer the reader to \cite{GKL} for an extensive list of recent literature on strong hypercontractivity in the holomorphic category, and related ideas (notably reverse hypercontractivity) in the subharmonic category.

If $\mu$ is a probability measure, the {\em entropy} functional $\mathrm{Ent}_\mu$ relative to $\mu$, defined on all sufficiently integrable positive test functions $g$, is
\[ \mathrm{Ent}_\mu(g) = \int g\ln\left(\frac{g}{\|g\|_1}\right)\,d\mu \]
where $\|g\|_1 = \|g\|_{L^1(\mu)}$.  (When $\|g\|_1=1$, so $g$ is a probability density, this gives the classical Shannon entropy.)   The {\bf logarithmic Sobolev inequality} is an energy-entropy functional inequality: a measure $\mu$ on $\R^n$ (or more generally on a Riemannian manifold) satisfies a log Sobolev inequality if, for some constant $c>0$ and for all sufficiently smooth positive test functions $f$,
\begin{equation} \label{LSI} \tag{LSI} \mathrm{Ent}_\mu(f^2) \le c \int |\nabla f|^2\,d\mu. \end{equation}
Making the substitution $g=f^2$ gives the equivalent form $\mathrm{Ent}_\mu(g)\le \frac{c}{4}\int |\nabla g|^2/g\,d\mu$, the integral on the right defining the {\em Fisher information} of $g$ relative to $\mu$.  In this form, the inequality was first discovered for the standard normal law $\mu$ on $\R$ by Stam in \cite{Stam}.  It was rediscovered and named by Gross in \cite{g4}, proved for standard Gaussian measures on $\R^n$ with sharp constant $c=2$.  Over the past four decades, it has become  an enormously powerful tool making fundamental contributions to geometry and global analysis \cite{Bakry 1,Bakry 2, Bakry Emery,Bobkov Houdre,Bobkov Ledoux,Davies 1,Davies 2,Davies Simon,Ledoux 0,Ledoux 2,Ledoux 5}, statistical physics \cite{Holley Stroock,Yau 1,Yau 2,Zegarlinski}, mixing times of Markov chains \cite{Bobkov Tetali,Diaconis Saloff-Coste,Guionnet Zegarlinski}, concentration of measure and optimal transport \cite{Ledoux 1,Ledoux 3,Villani}, random matrix theory \cite{AGZ,Ledoux 4,Zimmermann}, and many others.

\medskip

Gross discovered the log Sobolev inequality through his work in constructive quantum field theory, particularly relating to Nelson's hypercontractivity estimates \cite{n}.  In fact, Gross showed in \cite{g4} that the log Sobolev inequality (\ref{LSI}) is equivalent to hypercontractivity.  Later, in \cite{j,j2}, Janson discovered a stronger form of hypercontractivity that holds for {\em holomorphic} test functions.

\begin{theorem}[Janson \cite{j}] \label{thm Janson} If $\mu$ is the standard Gaussian measure on $\C^n$, and $0<p\le q<\infty$, then for all holomorphic functions $f\in L^p(\C^n,\mu)$, $\|f(e^{-t}\,\cdot\,)\|_q \le \|f\|_p$ for $t \ge \frac12\ln\frac{q}{p}$; for $t < \frac12\ln\frac{q}{p}$, the dilated function $f(e^{-t}\,\cdot\,)$ is not in $L^q(\C^n,\mu)$.
\end{theorem}

\begin{remark} Nelson's hypercontractivity estimates \cite{n} involve the semigroup $e^{-tA_\mu}$, where $A_\mu$ is the Dirichlet form operator for the measure $\mu$: $\int |\nabla f|^2\,d\mu = \int \overline{f}A_\mu f\,d\mu$. If $d\mu = \rho\,dx$ has a smooth density $\rho$, integration by parts shows that $A_\mu = -\Delta - (\nabla\rho/\rho)\cdot\nabla$, and so when applied to holomorphic (hence harmonic) functions, $e^{-tA_\mu}$ is the flow of the vector field $\nabla\rho/\rho$.  For the standard Gaussian measure, this is just the coordinate vector field $x$, the infinitesimal generator of dilations $Ef(x) = x\cdot\nabla f(x)$, also known as the {\em Euler operator}.  The perspective of this paper, like
%%%%%%%%%%%%%%%
its
%%%%%%%%%%%%%%Ù
predecessor \cite{GKL}, is that the strong hypercontractivity theorem is essentially about the dilation semigroup $f\mapsto f(e^{-t}\,\cdot\,)$, independent of the underlying measure.
\end{remark}

Janson's strong hypercontractivity
%%%%%%%%%%%%
differs
%%%%%%%%%%%%%%%%%%%%
from Nelson's hypercontractivity in two important ways: first, the time-to-contraction is smaller, $\frac12\ln\frac{q}{p}$ as opposed to the larger Nelson time $\frac12\ln\frac{q-1}{p-1}$, and second, the theorem applies even in the regime $0<p,q<1$ where the $L^p$ ``norms'' are badly-behaved.  Nevertheless, in \cite{g}, Gross showed that Janson's theorem is also a consequence of the same log Sobolev inequality (\ref{LSI}); moreover, he generalized this implication considerably to complex manifolds (equipped with sufficiently nice measures).  The reverse implication, however, was not established: the proof requires (\ref{LSI}) to hold for non-holomorphic functions (in particular of the form $|f|^{p/2}$).  We refer the reader to \cite{GKL} for an extensive list of recent literature on strong hypercontractivity in the holomorphic category, and related ideas (notably reverse hypercontractivity) in the subharmonic category.

\medskip

The aim of the present paper is to prove an {\em intrinsic} equivalence of strong hypercontractivity and a log Sobolev inequality.  The starting point is a generalization of Theorem \ref{thm Janson} beyond the holomorphic category.  A function on $\R^n$ is {\em log-subharmonic} ($\mathrm{LSH}$ for short) if $\ln |f|$ is subharmonic; holomorphic functions are prime examples.  In \cite{GKL}, we proved that Theorem \ref{thm Janson} holds in the larger class $\mathrm{LSH}$, for the Gaussian measure and several others.  We also established a weak connection to a {\bf strong log Sobolev inequality}.

\begin{definition} \label{def sLSI} A measure $\mu$ on $\R^n$ satisfies a {\bf strong logarithmic Sobolev inequality} if there is a constant $c>0$ so that, for non-negative $g\in\LSH$ sufficiently smooth and integrable,
\begin{equation} \label{sLSI} \tag{sLSI} \mathrm{Ent}_\mu(g) \le \frac{c}{2} \int Eg\,d\mu.
\end{equation}
\end{definition}
Inequality (\ref{sLSI}) could be written equivalently in the form $\mathrm{Ent}_\mu(f^2) \le c \int f Ef\,d\mu$; we will use it in $L^1$-form throughout.  In \cite{GKL}, we showed the strong log Sobolev inequality holds for the standard Gaussian measure on $\R^n$, with constant $c=1$ (half the constant from (\ref{LSI})), and conjectured that (\ref{sLSI}) is equivalent in greater generality to the following form of Janson's strong hypercontractivity.

\begin{definition} \label{def sHC} A measure $\mu$ on $\R^n$ satisfies the property of {\bf strong hypercontractivity} if there is a constant $c>0$ so that, for $0<p\le q<\infty$ and for every $f\in L^p(\mu)\cap\LSH$, we have
\begin{equation} \label{sHC} \tag{sHC} \|f(r\,\cdot\,)\|_{L^q(\mu)} \le \|f\|_{L^p(\mu)} \quad \text{if} \quad 0<r\le (p/q)^{c/2}.
\end{equation}
\end{definition}

\begin{remark} The statement in Definition \ref{def sHC} is given in multiplicative notation rather than additive, with $r=e^{-t}$ scaling the variable.  It would appear more convenient to use the constant $c$ instead of $\frac{c}{2}$ in (\ref{sLSI}) and (\ref{sHC}).  We choose to normalize with $\frac{c}{2}$ for historical reasons: Gross's equivalence of the log Sobolev inequality and Nelson's hypercontractivity equates $c$ in (\ref{LSI}) to $\frac{c}{2}$ scaling the time to contraction.
\end{remark}

\begin{notation} For a function $f$ on $\R^n$ and $r\in[0,1]$, $f_r$ denotes the function $f_r(x) = f(rx)$.  \end{notation}

\bigskip

\subsection{Main Results}

%To begin, we prove that (\ref{sLSI}) and (\ref{sHC}) are equivalent for compactly-supported measures.
%
%\begin{theorem}\label{th compact}
%  Let a  measure $\mu$ on $\R^n$ have a compact support. Then $\mu$ satisfies strong hypercontractivity (\ref{sHC}) with constant $c>0$ if and only if $\mu$  satisfies the strong logarithmic Sobolev inequality (\ref{sLSI}) with the same constant $c$.
%\end{theorem}
%\noindent It is atypical for compactly-supported measures to satisfy the standard log Sobolev inequality (\ref{LSI}); the strong log Sobolev inequality, which is required only to hold for the restrictive class $\mathrm{LSH}$ of log-subharmonic functions, is more forgiving.  Indeed, Proposition \ref{prop centered cc} below (generalizing \cite[Proposition 4.2]{GKL}) establishes that (\ref{sLSI}) holds true for any compactly-supported probability measure with mean $0$, with constant $c\le 2$.

In \cite{GKL}, we showed that (\ref{sHC}) implies (\ref{sLSI}) in the special case that the measure $\mu$ is compactly supported.  Our first result is the converse.

\begin{theorem} \label{t.comp.supp} Let $\mu$ be a compactly supported measure on $\R^n$.  Suppose that $\mu$ satisfies (\ref{sLSI}) for all sufficiently smooth functions $g\in\LSH(\R^n)$.  Then $\mu$ satisfies (\ref{sHC}) for {\em all} functions $f\in\LSH(\R^n)$.
\end{theorem}
\begin{remark} \label{r.comp.supp} We emphasize here that the domains in the equivalence consist of log-subharmonic functions a priori defined on all of $\R^n$, not just on the support of $\mu$.  Indeed, the dilation semigroup is not well-defined if this is not satisfied.  In fact, it is not hard to see that this result extends to log-subharmonic functions defined on any
%%%%%%%%%%%%%%%%start-shaped 
star-shaped
%%%%%%%%%%%%%%%%%%%%%%%%%%
open region containing the support of $\mu$.
\end{remark}

Theorem \ref{t.comp.supp} and its converse have non-trivial applications: for example, Proposition \cite[Proposition 4.2]{GKL} implies that (\ref{sLSI}) holds true for any compactly supported symmetric measure on $\R$, with constant $c\le 2$.  Nevertheless, it excludes the standard players in log Sobolev inequalities, most notably Gaussian measures.  In \cite[Theorem 5.8]{GKL}, we proved directly that (\ref{sLSI}) holds true for the standard Gaussian measure on $\R^n$, with best constant $c=1$.  This was proved directly from (\ref{LSI}), and relied heavily on the precise form of the Gaussian measure; a direct connection to strong hypercontractivity (also proved for the Gaussian measure in \cite[Theorem 3.2]{GKL}) was not provided.  That connection, for a wide class of measures, is the present goal.

\medskip

The technicalities involved in establishing the equivalence of (\ref{sLSI}) and (\ref{sHC}) are challenging because of the rigidity of the class $\mathrm{LSH}$: standard cut-off approximations needed to use integrability arguments in the proof are unavailable for subharmonic functions.  To amend this, we use a fundamentally different approximation technique: the {\bf dilated convolution} introduced in \cite{GGS,Gross Grothaus}, and developed in Section \ref{sect continuity dilated convolution} below.  In \cite{GGS}, the authors provided a local condition on the density of $\mu$ under which this operation is bounded on $L^p(\mu)$ (amounting to a bound on the Jacobian derivative of the translation and dilation).  Here we present alternative conditions, which require little in terms of the local behaviour of the measure (they are essentially growth conditions near infinity) and achieve the same effect.

\begin{definition} \label{def Euclidean regular} Let $p>0$ and let $\mu$ be a positive measure on $\R^n$ with density $\rho$.  Say that $\mu$ (or $\rho$) is {\bf Euclidean exponential type} $p$ if $\rho(x)>0$ for all $x$ and if the following two conditions hold:
\begin{align}
\sup_x \sup_{|y|\le s}|x|^p \frac{\rho(ax+y)}{\rho(x)}&<\infty \quad \text{for} \  \text{any} \quad a>1, \; s\ge 0 \label{E.R.1} \\
\sup_x \sup_{1< a< 1+\e} \frac{\rho(ax)}{\rho(x)} &<\infty \quad \text{for} \ \text{some} \quad \e>0. \label{E.R.2}
\end{align}
If $\mu$ is Euclidean exponential type $0$, we say it is {\bf Euclidean regular}.
\end{definition}

The  terminology derives from the fact that conditions (\ref{E.R.1}) and (\ref{E.R.2}) insist that the Euclidean group acts on $\rho$ in a controlled manner; exponential type refers to the growth condition involving $|x|^p$ (indeed, for $p>0$ the measure must have tails that decay faster than any polynomial to be Euclidean exponential type $p$).  For any probability measure $\mu$ with strictly positive density $\rho$, denote for $a\ge 1$ and $p,s\ge0$
\begin{equation} \label{eq E.R. constant} C^p_\mu(a,s) = C^p_\rho(a,s) \equiv \sup_x \sup_{|y|\le s} |x|^p\frac{\rho(ax+y)}{\rho(x)}. \end{equation}
Then the condition that $\mu$ is Euclidean exponential type $p$ is precisely that $C^p_\mu(a,s)<\infty$ for each $a>1$ and $s\ge 0$, and $C^0_\mu(a,0)$ is uniformly bounded for $a$ close to $1$.  It is clear from the definition that $C^p_\mu(a,s)$ is an increasing function of $s$.  Moreover, if $\mu$ is Euclidean  exponential type $q$ then it is Euclidean  exponential type $p$ for any $p<q$.  For convenience, we will often write $C_\mu$ for $C_\mu^0$.

\begin{example} \label{examples E.R.}
On $\R$, the densities $(1+x^2)^{-\alpha}$ for $\alpha>\frac12$ are Euclidean regular.
 On $\R^n$ the densities  $e^{-c|x|^a}$ with $a,c>0$ are Euclidean exponential type $p$ for all $p>0$.
\end{example}
\noindent More examples and properties that prove the Euclidean regular measures form a rich class are given in Section \ref{sect properties E.R.}.  The purpose of introducing this class at present is its utility in proving a density theorem for an appropriate class of Sobolev-type spaces.

\begin{definition} \label{def LpE} Let $\mu$ be a measure on $\R^n$, and let $p>0$.  Define the Sobolev space $L^p_E(\mu)$ to consist of those weakly-differentiable functions $f\in L^p(\mu)$ for which $Ef\in L^p(\mu)$.  To be clear: $Ef(x) = \sum_{j=1}^n x_j u_j(x)$, where $u_j$ is the function (posited to exist) satisfying
\[ -\int \del_j \varphi\, f\,dx = \int \varphi\, u_j\,dx \]
for any $\varphi\in C_c^\infty(\R^n)$, where $dx$ denotes Lebesgue measure.
\end{definition}

Standard techniques, involving approximation by $C_c^\infty$ functions, show that $L^p_E$ is dense in $L^p$ for reasonable measures.  However, our goals here involve approximation of log-subharmonic functions, and the usual cut-off approximations fail to preserve subharmonicity.  An alternative approach is to use a convolution approximate identity procedure, as is readily available for Lebesgue measure.  The problem is that, for a given bump function $\varphi$, the operation $f\mapsto f\ast\varphi$ is typically unbounded on $L^p(\mu)$ when $\mu$ is not Lebesgue measure.  Indeed, for $L^p$ of Gaussian measure, even the translation $f\mapsto f(\;\cdot\,+y)$ is unbounded.  The problem is that the convolution can shift mass in from near infinity.  One might hope to dilate this extra mass back out near infinity, to preserve $p$-integrability; thus the dilated convolution $f\mapsto (f\ast\varphi)_r$.  Section \ref{sect continuity dilated convolution} shows that this operation behaves well in $L^p$ spaces of 
Euclidean regular measures; it also preserves the cone $\mathrm{LSH}$.

\medskip

The main technical theorem of this paper is the following Sobolev density theorem.

\begin{theorem} \label{thm Sobolev density} Let $p\in[0,\infty)$, and let $\mu$ be a Euclidean exponential type $p$ probability measure on $\R^n$.  Then the cone $\LSH\cap L^p_E(\mu)$ is dense in the cone $\LSH\cap L^p(\mu)$. \end{theorem}

\noindent Using Theorem \ref{thm Sobolev density}, we will prove the equivalence of (\ref{sLSI}) and (\ref{sHC}), the latter in a nominally weaker form.

\begin{definition} \label{def Lp closure} Let $\mu$ be a probability measure on $\R^n$, and let $0<p<q<\infty$.  Denote by $\mathrm{LSH}^{p<q}(\mu)$ the closure of $L^q(\mu)\cap\mathrm{LSH}$ in $L^p(\mu)\cap\mathrm{LSH}$. \end{definition}

\noindent For any probability measure, there is a common dense subspace ($L^\infty$) for all the full $L^q$-spaces, $q>0$, and so the closure of $L^q$ in $L^p$ is all of $L^p$ for $p<q$; the proof uses cut-offs that do not respect subharmonicity, and indeed, there are no non-constant bounded subharmonic functions.  In \cite{g}, Gross showed that, under certain conditions on a measure $\mu$ on a complex manifold (in terms of its Dirichlet form operator $d^\ast d$), in the presence of a full log Sobolev inequality (\ref{LSI}), there is a common dense subspace for all {\em holomorphic} $L^q$ spaces of $\mu$.  In the present context of logarithmically-subharmonic functions, no such technology is known, and we will content ourselves with the spaces $\mathrm{LSH}^{p<q}$.  We will consider the nature of the spaces in a future publication.

%\medskip
%
%One further regularity condition on the measure $\mu$ is required for the integrability arguments in what follows.
%
%\begin{definition} \label{def rotationally quasi-invariant} Let $\mu$ be a Borel probability measure on $\R^n$.  Denote by $\tilde{\mu}$ its rotational average: for a Borel set $B$, $\tilde{\mu}(B) = \int_{O(n)} \mu(\alpha\cdot B)\,\vartheta(d\alpha)$ where $\vartheta$ denotes the Haar measure on the orthogonal group $O(n)$.  The measure $\mu$ is called {\bf rotationally quasi-invariant} if there are constants $C_1,C_2>0$ so that $C_1 \tilde{\mu} \le \mu \le C_2 \tilde{\mu}$. \end{definition}
%Of course, rotationally invariant measures (for which $\mu = \tilde{\mu}$) are rotationally quasi-invariant. Our proofs do not require strict invariance; but it is also insufficient to assume Euclidean regularity for some regularization techniques to work.

\medskip

This brings us to our main theorem: the equivalence of (\ref{sLSI}) and (\ref{sHC}) for logarithmically subharmonic functions.

\begin{theorem} \label{thm sLSI <=> sHC} Let $\mu$ be an $O(n)$-invariant probability measure on $\R^n$.
\begin{itemize}
\item[(1)] If $\mu$ is Euclidean exponential type $p$ for {\em all} $p>0$, and if $\mu$ satisfies the strong log Sobolev inequality (\ref{sLSI}), then $\mu$ satisfies strong hypercontractivity (\ref{sHC}) in the spaces $\mathrm{LSH}^{p<q}(\mu)$: for $0<p\le q<\infty$ and $f\in \mathrm{LSH}^{p<q}$, $\|f_r\|_q \le \|f\|_p$\; for $0<r\le (p/q)^{c/2}$.
\item[(2)] If $\mu$ is Euclidean exponential type $p$ for {\em some} $p>1$, and if $\mu$ satisfies (\ref{sHC}) in the above sense, then $\mu$ satisfies the strong log Sobolev inequality (\ref{sLSI}):
\[ \mathrm{Ent}_\mu(g)\le \frac{c}{2}\int Eg\,d\mu \]
for all $g\in\mathrm{LSH}\cap L^1_E(\mu)\cap C^1(\R^n)$.
\end{itemize}
\end{theorem}

\begin{remark} \label{rk O(n)-invariant} The global assumption of rotational-invariance in Theorem \ref{thm sLSI <=> sHC} is actually quite natural in this situation.   The functional $g\mapsto \int Eg\,d\mu$ on the right-hand-side of our strong log Sobolev inequality is not generally positive, since the operator $E$ is not generally self-adjoint in $L^2(\mu)$; however, when $\mu$ is rotationally-invariant, this functional is positive on the cone $\LSH$, as pointed out in \cite[Proposition 5.1]{GKL}.
\end{remark}

\noindent We emphasize that Theorem \ref{thm sLSI <=> sHC} is {\em intrinsic}.  While the two directions of the theorem require slightly different assumptions on the applicable measures, the implications between (\ref{sLSI}) and (\ref{sHC}) both stay within the cone $\mathrm{LSH}$ of log-subharmonic functions.  This is the main benefit of extending Janson's strong hypercontractivity theorem from holomorphic functions to this larger class, and restricting the log-Sobolev inequality to it: here, the two are precisely equivalent.

%\medskip
%
%The remainder of this paper is organized as follows. In Section \ref{section alternative}, we use the semigroup property of $f\mapsto f_r$ to reformulate strong hypercontractivity in a form easier to compute with.  Section \ref{sect properties E.R.} shows that the class of measures under consideration (Euclidean regular or exponential type) is quite rich: it is closed under bounded perturbations, convex combinations, products, and convolutions.  Section \ref{sect continuity dilated convolution} deals with the continuity of the dilated convolution operation $f\mapsto (f\ast\varphi)_r$ on $L^p(\mu)$, and Section \ref{sect proof Sobolev density} proves Theorem \ref{thm Sobolev density}, that the Sobolev-type space $L^p_E(\mu)$ is dense in $L^p(\mu)$ through the cone $\LSH$ of log-subharmonic functions.
%
%\ldots

\subsection{Alternative Formulation of sHC} \label{section alternative}

The following equivalent characterization of strong hypercontractivity will be useful in what follows.

\begin{proposition} \label{prop SHC increasing} Fix $c>0$ and let $q(r)$ denote the function $q(r) = r^{-2/c}$. A measure $\mu$ satisfies strong hypercontractivity (\ref{sHC}) if and
%%%%%%%%%%%%%%%%%%%%%%%%onld 
only
%%%%%%%%%%%%%%%%%%%%%%%%%%%%%%%%%%%
if for each function $f\in L^1(\mu)\cap\LSH$,
\[ \|f_r\|_{q(r)} \le \|f\|_1 \quad \text{and} \quad \|f_r\|_1 \le \|f\|_1, \quad \text{for }\; r\in(0,1]. \]
\end{proposition}

\begin{remark} Similarly, the form of strong hypercontractivity given in Theorem \ref{thm sLSI <=> sHC} is equivalent to the same inequalities above holding for all $f$ in the nominally smaller space $\mathrm{LSH}^{1<q(r)}$. \end{remark}

\noindent For the proof, it is useful to note that the class $\mathrm{LSH} $ is closed under $f\mapsto f^p$ for any $p>0$.

\begin{proof} First, suppose (\ref{sHC}) holds with constant $c$.  The case $p=q=1$ yields $\|f_r\|_1 \le \|f\|_1$ for $0<r\le (p/q)^{2/c} = 1$.  More generally, by (\ref{sHC}), $\|f_r\|_q \le \|f\|_1$ whenever $0<r\le (1/q)^{c/2}$; i.e. whenever $q\ge r^{-c/2} = q(r)$.  In particular, it follows that $\|f_r\|_{q(r)} \le \|f\|_1$ as claimed.

\medskip

Conversely, suppose the above conditions hold true.  Fix $q\ge p>0$ and let $f\in L^p(\mu)\cap\LSH$.  Then $f^p\in L^1(\mu)\cap\LSH$, and so by assumption we have $\|(f^p)_r\|_{q(r)} \le \|f^p\|_1$ for $0<r\le 1$.  Since $(f^p)_r = (f_r)^p$, it follows immediately that $\|f_r\|_{p\cdot q(r)}^p \le \|f\|_p^p$.  Setting $q=p\cdot q(r)$ and solving for $r$, we have $r=r(p,q)\equiv (p/q)^{c/2}$, and so we have proved the equality case of (\ref{sHC}).  Finally, suppose that $r'\le r(p,q)= (p/q)^{c/2}$; then there is $s\in(0,1]$ so that $r' = s\cdot r(p,q)$.  Dilations form a multiplicative semigroup, so $f_{r'} = (f_{r(p,q)})_s$.  We have just proved that $f_{r(p,q)}\in L^q$, and hence $(f_{r(p,q)})^q$ is in $L^1(\mu)$.  Therefore, by assumption, $\|[(f_{r(p,q)})^q]_s\|_1 \le \|(f_{r(p,q)})^q\|_1$; unwinding this yields
\[ \|f_{r'}\|_q^q = \|(f_{r(p,q)})_s\|_q^1 = \|[(f_{r(p,q)})_s]^q\|_1 =  \|[(f_{r(p,q)})^q]_s\|_1 \le \|(f_{r(p,q)})^q\|_1 = \|f_{r(p,q)}\|_q^q\le \|f\|_p^q \]
by the equality case, thus proving (\ref{sHC}).
\end{proof}

\begin{remark} In fact, (\ref{sHC}) implies the putatively stronger statement that $r\mapsto \|f_r\|_{q(r)}$ is non-decreasing on $[0,1]$; however, the weaker form presented above is generally easier to work with. \end{remark}

\subsection{Convolution property}
We will use the convolution operation to prove the Sobolev density theorem at the heart of this paper, as well as Theorem \ref{t.comp.supp}.  We begin by showing that this operation preserves the cone $\LSH$.

\begin{lemma} \label{lemma convolution preserves LSH} Let $f\in\LSH$.  Let $\varphi\ge 0$ be a $C^\infty_c$ test function.  Then $f\ast\varphi\in \LSH\cap C^\infty$.
\end{lemma}

\begin{proof} Since $f\in\LSH$, $f\ge 0$ and $\ln f$ is subharmonic.  In particular, $\ln f$ is upper semi-continuous and locally bounded above, and so the same holds for $f$.  Thus $f$ is locally bounded and measurable; thus $f\ast\varphi$ defines an $L^1_{\mathrm{loc}}\cap C^\infty$ function.  We must show it is $\LSH$.

\medskip

Any subharmonic function is the decreasing limit of a sequence of $C^\infty$ subharmonic functions, cf. \cite[Appendix 1, Proposition 1.15]{Lelong Gruman}.  Applying this to $\ln f$, there is a sequence $f_n\in\LSH\cap C^\infty$ such that $f_n\downarrow f$.   Let $g_n = f_n+\frac1n$; so $g_n$ is strictly positive, and $g_n\downarrow f$.  Since $\varphi$ is $\ge 0$, it follows from the Monotone Convergence Theorem that $g_n\ast\varphi \downarrow f\ast\varphi$ pointwise.

\medskip

Now, $(g_n\ast\varphi)(x) = \int_{\R^n} g_n(x-\omega)\varphi(\omega)\,d\omega$.  Since translation and positive dilation preserve the cone $\LSH$, the function $x\mapsto g_n(x-\omega)\varphi(\omega)$ is continuous and $\LSH$ for each $\omega$.  Moreover, the function $\omega\mapsto g_n(x-\omega)\varphi(\omega)$ is continuous and bounded.  Finally, for small $r$, $\sup_{|t-x|\le r} g_n(t-\omega)\varphi(\omega) \le \|\varphi\|_\infty \sup_{|t|\le |x|+r} g_n(t)$ is bounded uniformly in $\omega$.  It follows from \cite[Lemma 2.4]{GKL} that $g_n\ast\varphi$ is $\LSH$.  (The statement of that lemma apparently requires the supremum to be uniform in $x$ as well, but this is an overstatement; as the proof of the lemma clearly
%%%%%%%%%%%%%%%%%,
shows, only uniformity in $\omega$ is required).

\medskip

Thus, $f\ast\varphi$ is the decreasing limit of strictly positive $\LSH$ functions $g_n\ast\varphi$.  Applying the Monotone Convergence Theorem to integrals of $\ln(g_n\ast\varphi)$ about spheres now shows that $\ln(f\ast\varphi)$ is subharmonic, so $f\ast\varphi\in\LSH$ as claimed.
\end{proof}

\subsection{Compactly Supported Measures}

This section is devoted to the proof of Theorem \ref{t.comp.supp}.  It follows the now-standard Gross proof of such equivalence: differentiating hypercontractivity at the critical time yields the log Sobolev inequality, and vice versa.  The technical issues related to differentiating under the integral can be dealt with fairly easily in the case of a compactly supported measure; the remainder of this paper develops techniques for handling measures with non-compact support.  The forward direction of the theorem, that (\ref{sHC}) implies (\ref{sLSI}) for compactly supported measures, is \cite[Theorem 5.2]{GKL}, so we will only include the proof of the reverse direction here.
%%%%%%%%%%%%%%%%%%%%%%%%%%%%%%%(although both directions are essentially the same).
\medskip

\begin{proof}[Proof of Theorem \ref{t.comp.supp}]
By assumption, (\ref{sLSI}) holds for sufficiently smooth and integrable functions; here we interpret that precisely to mean $\mathrm{Ent}_\mu(g)\le \frac{c}{2}\int Eg\,d\mu$ for all $g\in C^1(\R^n)$ for which both sides are finite.  Fix $f\in L^1(\mu)\cap\mathrm{LSH}\cap C^1$.  Utilizing Proposition \ref{prop SHC increasing}, we must consider the function $\alpha(r) = \|f_r\|_{q(r)}$ where $q(r) = r^{-2/c}$.  Let $\beta(r) = \alpha(r)^{q(r)} = \int f(rx)^{q(r)}\,\mu(dx)$ and set $\beta_x(r) = f(rx)^{q(r)}$ so that $\beta(r) = \int \beta_x(r)\,\mu(dx)$.  Then,

\[ \frac{\del}{\del r}\ln \beta_x(r) = q'(r)\ln f(rx)+\frac{q(r)}{f(rx)}x\cdot\nabla f(rx). \]
Since $q'(r) = -\frac{2}{rc}q(r)$, and since $x\cdot \nabla f(rx) = \frac1r(Ef)_r(x)= \frac1rE(f_r)(x)$, we have
\begin{equation} \label{eq compact beta'} \frac{\del}{\del r}\beta_x(r) = -\frac{2}{rc}f_r(x)^{q(r)}\ln f_r(x)^{q(r)} + \frac1rq(r)f_r(x)^{q(r)-1}(Ef_r)(x). \end{equation}
Fix $0<\e<1$.  As $f$ is $C^1$, the function (of $x$) on the right-hand-side of (\ref{eq compact beta'}) is uniformly bounded for $r\in(\e,1]$ and $x\in\supp\mu$ (due to compactness).  The Dominated Convergence Theorem thus allows differentiation under the integral, and so
\begin{equation} \label{eq compact beta' 2} \beta'(r) = \int \frac{\del}{\del r}\beta_x(r)\,\mu(dx). \end{equation}
Thus, since $\alpha(r) = \beta(r)^{1/q(r)}$ and $\beta(r)>0$, it follows that $\alpha$ is $C^1$ on $(\e,1]$ and the chain rule yields
\begin{equation} \label{eq compact alpha'} \alpha'(r) = \frac{\alpha(r)}{q(r)\beta(r)}\frac{2}{rc}\left[\beta(r)\ln\beta(r) + \frac{rc}{2}\beta'(r)\right]. \end{equation}
From (\ref{eq compact beta'}) and (\ref{eq compact beta' 2}), the quantity in brackets is
\begin{align}\nonumber &\int f_r^{q(r)}\,d\mu\cdot \ln\int f_r^{q(r)}\,d\mu + \frac{rc}{2}\int\left(-\frac{2}{rc}f_r^{q(r)}\ln f_r^{q(r)} + \frac1r q(r)f_r^{q(r)-1}Ef_r\right)\,d\mu \\ \nonumber
=& \int f_r^{q(r)}\,d\mu\cdot \ln\int f_r^{q(r)}\,d\mu - \int f_r^{q(r)}\ln f_r^{q(r)}\,d\mu + q(r)\frac{c}{2}\int f_r^{q(r)-1}Ef_r\,d\mu \\ \label{e.Ent.ineq}
=& -\mathrm{Ent}_\mu(f_r^{q(r)}) + \frac{c}{2}\int E(f_r^{q(r)})\,d\mu,
\end{align}
where the equality in the last term follows from the chain rule.

Since $f\in C^1$, it is bounded on the compact set $\supp\mu$, and so are all of its dilations $f_r$.  Hence, both terms in (\ref{e.Ent.ineq}) are finite, and so by the assumption of the theorem, this term is $\ge 0$.  From (\ref{eq compact alpha'}), we therefore have $\alpha'(r)\ge 0$ for all $r>\epsilon$.  Since this is 
%%%%%%%%%%%%%%%%%%%%%%%%%trull
true
%%%%%%%%%%%%%%%%%%%%%%%%%%%%%%%%%%%%%
for each $\epsilon>0$, it holds true for $r\in(0,1]$.  This verifies the first inequality in Proposition \ref{prop SHC increasing}.  For the second, we use precisely the same argument to justify differentiating under the integral to find
\[ \frac{\del}{\del r} \|f_r\|_1 = \int \frac{\del}{\del r}f_r(x)\,\mu(dx) = \frac1r\int Ef_r(x)\,\mu(dx) \ge \frac{2}{cr}\mathrm{Ent}_\mu(f_r) \ge 0 \]
by the assumption of (\ref{sLSI}).  This concludes the proof for $f\in C^1$. 

Now, if  $f\in L^1(\mu)\cap\mathrm{LSH}$, we consider a smooth approximate identity sequence  $\varphi_k$. The inequalities in Proposition  \ref{prop SHC increasing} hold for $f\ast\varphi_k$ by the first part of the proof and Lemma \ref{lemma convolution preserves LSH}. Note by simple change of variables that $(f\ast\varphi_k)_r = f_r \ast (r^n\varphi_k)_r$, and that $(r^n\varphi_k)_r$ is also an approximate identity sequence. The function $f_r$ is $\mathrm{LSH}$, so it is upper semi-continuous and consequently locally bounded. Thus $f_r\in L^{q(r)}$ and $(f\ast\varphi_k)_r$ converges to  $ f_r$ in  $L^{q(r)}$.  This concludes the proof.
\end{proof}

\section{Density results through $\LSH$ functions} \label{sect density}

\subsection{Properties of Euclidean regular measures} \label{sect properties E.R.}

In this section, we show several closure properties of the class of Euclidean regular measures (of any given exponential type $p\in[0,\infty)$): it is closed under bounded perturbations, convex combinations, product, and convolution.  Throughout, we use $\mu_i$ ($i=1,2$) to stand for such measures, and $\rho_i$ to stand for their densities.

\begin{proposition} \label{prop closure 1} Let $\mu_1$ and $\mu_2$ be positive measures on $\R^n$, and suppose $\mu_1$ is Euclidean exponential type $p\in[0,\infty)$.  If there are constants $C,D>0$ such that $C\mu_1 \le \mu_2 \le D\mu_1$, then $\mu_2$ is also Euclidean exponential type $p$. \end{proposition}

\begin{proof} The assumption is that $C\rho_1 \le \rho_2 \le D\rho_1$.  Let $\e>0$ be such that $\sup_{1<a<1+\e} C^0_{\rho_1}(a,0) < \infty$.  Then for any such $a$,
\[ \frac{\rho_2(ax)}{\rho_2(x)} \le \frac{D \rho_1(ax)}{C\rho_1(x)} \le \frac{D}{C}C^0_{\rho_1}(a,0) \]
for all $x$; thus $C^0_{\rho_2}(a,0) \le \frac{D}{C}C^0_{\rho_1}(a,0)$, and so $\sup_{1<a<1+\e}C^0_{\rho_2}(a,0)<\infty$.  Similarly, for $x,y\in\R^n$ and $a>1$,
\[ |x|^p\frac{\rho_2(ax+y)}{\rho_2(x)} \le |x|^p \frac{D\rho_1(ax+y)}{C\rho_1(x)} \le \frac{D}{C}C^p_{\rho_1}(a,|y|) \]
and so $C^p_{\rho_2}(a,s) \le \frac{D}{C}C^p_{\rho_1}(a,s) <\infty$.
\end{proof}

\begin{proposition} \label{prop closure 2} Let $\mu_1$ and $\mu_2$ be Euclidean regular measures of exponential type $p\in[0,\infty)$.  For any $t\in[0,1]$, $\mu = (1-t)\mu_1 + t\mu_2$ is Euclidean exponential type $p$. \end{proposition}

\begin{proof} Let $\e>0$ be such that $\sup_{1<a<1+\e} C^0_{\rho_i}(a,0)<\infty$ for $i=1,2$.  Let $\rho$ be the density of $\mu$.  Then for any $x\in\R^n$,
\begin{align*} \rho(ax) = (1-t)\rho_1(ax) + t\rho_2(ax) &\le (1-t)C^0_{\rho_1}(a,0) \rho_1(x) + tC^0_{\rho_2}(a,0) \rho_2(x) \\
&\le \max\{C^0_{\rho_1}(a,0),C^0_{\rho_2}(a,0)\}\rho(x) \end{align*}
and so $C^0_{\rho}(a,0) \le \max\{C^0_{\rho_1}(a,0),C^0_{\rho_2}(a,0)\}$ is uniformly bounded for $1<a<1+\e$, as required.  Similarly, for $x,y\in\R^n$ and $a>1$,
\begin{align*} |x|^p\rho(ax+y) \le (1-t) |x|^p\rho_1(ax+y) + t |x|^p\rho_2(ax+y) &\le (1-t) C^p_{\rho_1}(a,|y|)\rho_1(x) + t C^p_{\rho_2}(a,|y|)\rho_2(x) \\
&\le \max\{C^p_{\rho_1}(a,|y|),C^p_{\rho_2}(a,|y|)\}\rho(x)
\end{align*}
which shows that $C^p_{\rho}(a,s) \le \max\{C^p_{\rho_1}(a,s),C^p_{\rho_2}(a,s)\}<\infty$ for $a\ge 1$ and $s\ge 0$.
\end{proof}

\begin{proposition} Let $p\in[0,\infty)$, let $\mu_1$ be a Euclidean exponential type $p$ measure on $\R^{n_1}$, and let Let $\mu_2$ be a Euclidean exponential type $p$ measure on $\R^{n_2}$.  Then the product measure $\mu_1\otimes\mu_2$ is Euclidean exponential type $p$ on $\R^{n_1+n_2}$.
\end{proposition}

\begin{proof} For $i=1,2$ let $\rho_i$ be the density of $\mu_i$; then $\mu_1\otimes\mu_2$ has density $\rho_1\otimes\rho_2(x_1,x_2) = \rho_1(x_1)\rho_2(x_2)$.  Fix $\e>0$ so that $\sup_{1<a<1+\e}C^p_{\rho_i}(0,a)<\infty$ for $i=1,2$.  Then, letting $\mathbf{x}=(x_1,x_2)$,
\[ \rho_1\otimes\rho_2(a\mathbf{x}) = \rho_1(ax_1)\rho_2(ax_2) \le C^0_{\rho_1}(a,0)\,\rho_1(x_1)\cdot C^0_{\rho_2}(a,0)\,\rho_2(x_2) \]
and so $C^0_{\rho_1\otimes\rho_2}(a,0) \le C^0_{\rho_1}(a,0)\cdot C^0_{\rho_2}(a,0)$, meaning $\sup_{1<a<1+\e}C^0_{\rho_1\otimes\rho_2}(a,0)<\infty$. Similarly, for fixed $\mathbf{x},\mathbf{y}\in\R^{n_1+n_2}$ and $a>1$,
\[ |\mathbf{x}|^p\rho_1\otimes\rho_2(a\mathbf{x}+\mathbf{y}) = (|x_1|+|x_2|)^p\rho_1(ax_1+y_1)\rho_2(ax_2+y_2). \]
By elementary calculus, $(|x_1|+|x_2|)^p \le 2^{p-1}(|x_1|^p+|x_2|^p)$, and so we have
\[ |\mathbf{x}|^p\rho_1\otimes\rho_2(a\mathbf{x}+\mathbf{y}) \le 2^{p-1}|x_1|^p\rho_1(ax_1+y_1)\cdot \rho_2(ax_2+y_2) + 2^{p-1}\rho_1(ax_1+y_1)\cdot |x_2|^p\rho_2(ax_2+y_2). \] 
For the first term, we have $|x_1|^p\rho_1(ax_1+y_1) \le C^p_{\rho_1}(a,|y_1|)\rho_1(x_1)$ while $\rho_2(ax_2+y_2) \le C^0_{\rho_2}(a,|y_2|)$; for the second term, we have $\rho_1(ax_1+y_1) \le C^0_{\rho_1}(a,|y_1|)\rho_1(x_1)$ while $|x_2|^p\rho_2(ax_2+y_2) \le C^p_{\rho_2}(a,|y_2|)$.  If $|y|\le s$ then $|y_i|\le s$ for $i=1,2$.  All together, this shows that
\[ C^p_{\rho_1\otimes\rho_2}(a,s) \le 2^{p-1}\left[C^p_{\rho_1}(a,s)C^0_{\rho_2}(a,s) + C^0_{\rho_1}(a,s)C^p_{\rho_2}(a,s)\right] \]
which is finite since both $\rho_1,\rho_2$ are Euclidean exponential type $p$ (and hence also Euclidean regular).  This proves the proposition.
\end{proof}

\begin{proposition} Let $\mu_1$ and $\mu_2$ be positive measures on $\R^n$, each of Euclidean exponential type $p\in[0,\infty)$.  Then $\mu_1\ast\mu_2$ is Euclidean exponential type $p$.
\end{proposition}

%\begin{remark} Since a measure of Euclidean exponential type $p_1$ is also Euclidean exponential type $p_2$ for any $p_2\le p_1$, the proposition could be stated in the nominally stronger form: if $\mu_1$ is type $p_1$ and $\mu_2$ is type $p_2$, then $\mu_1\ast\mu_2$ is type $\min\{p_1,p_2\}$.  This is very much in line with the behavior of log Sobolev constants under convolution.
%\end{remark}

\begin{proof} Let $\rho_j$ be the density of $\mu_j$.  By assumption, for $i=1,2$ $C^{p_i}_{\rho_i}(a,s)<\infty$ for $a>1$ and $s\ge 0$, and there is $\e>0$ such that $\sup_{1<a<1+\e}C^{0}_{\rho_i}(a,0)<\infty$, cf.\ (\ref{eq E.R. constant}).  Then for $a\ge 1$ and $x\in\R^n$
\[ \rho_1\ast\rho_2(ax) = \int \rho_1(ax-u)\rho_2(u)\,du = a^n\int \rho_1(ax-av)\rho_2(av)\,dv. \]
By definition, $\rho_1(a(x-v)) \le C^0_{\rho_1}(a,0)\rho_1(x-v)$ and $\rho_2(av) \le C^0_{\rho_2}(a,0)\rho_2(v)$ for all $x,v$.  Thus
\[ \rho_1\ast\rho_2(ax) \le a^n C^0_{\rho_1}(a,0)\cdot C^0_{\rho_2}(a,0) \int \rho_1(x-v)\rho_2(v)\,dv = a^n C^0_{\rho_1}(a,0)\cdot C^0_{\rho_2}(a,0) \rho_1\ast\rho_2(x). \]
It follows that $C^0_{\rho_1\ast\rho_2}(a,0) \le a^n C^0_{\rho_1}(a,0)\cdot C^0_{\rho_2}(a,0)$, and hence
\begin{equation} \label{eq conv E.E.R. 1} \sup_{1<a<1+\e} C^0_{\rho_1\ast\rho_2}(a,0) \le (1+\e)^n \sup_{1<a<1+\e} C^0_{\rho_1}(a,0)\cdot\sup_{1<a<1+\e} C^0_{\rho_2}(a,0) <\infty \end{equation}
as required.  Similarly, for $x,y\in\R^n$ and $a>1$,
\[ |x|^{p} \rho_1\ast\rho_2(ax+y) = |x|^{p}\int \rho_1(ax+y-u)\rho_2(u)\,du = a^n\int |x|^{p} \rho_1(a(x-v)+y)\rho_2(av)\,dv. \]
Note (by elementary calculus) that $|x|^{p} \le 2^{p-1}(|x-v|^{p} + |v|^{p})$, and so
\[ |x|^{p} \rho_1\ast\rho_2(ax+y) \le 2^{p-1}a^n \left[\int |x-v|^{p}\rho_1(a(x-v)+y)\rho_2(av)\,dv + \int \rho_1(a(x-v)+y)|v|^{p}\rho(av)\,dv\right]. \]
In the first term, we have $|x-v|^{p}\rho_1(a(x-v)+y) \le C^p_{\rho_1}(a,|y|) \rho_1(x-v)$ and $\rho_2(av) \le C^0_{\rho_2}(a,0)$, and so
\[ \int |x-v|^{p}\rho_1(a(x-v)+y)\rho_2(av)\,dv \le C^p_{\rho_1}(a,|y|)\cdot C^0_{\rho_2}(a,0)\, \rho_1\ast\rho_2(x). \]
In the second term, we have $\rho_1(a(x-v)+y) \le C^0_{\rho_1}(a,|y|) \rho_1(x-v)$ and $|v|^p\rho(av) \le C^p_{\rho_2}(a,0)$, and so
\[ \int \rho_1(a(x-v)+y)|v|^p\rho(av)\,dv \le C^0_{\rho_1}(a,|y|)\cdot C^p_{\rho_2}(a,0)\,\rho_1\ast\rho_2(x). \]
All together, for any $s\ge |y|$, this gives
\begin{equation} \label{eq conv E.E.R. 2} C^p_{\rho_1\ast\rho_2}(a,s) \le 2^{p-1}a^n\left[ C^p_{\rho_1}(a,s)\cdot C^0_{\rho_2}(a,0) + C^0_{\rho_1}(a,s)\cdot C^p_{\rho_2}(a,0)\right] \end{equation}
which is finite since both $\rho_1$ and $\rho_2$ are Euclidean exponential type $p$ (and thus also Euclidean regular).  Equations (\ref{eq conv E.E.R. 1}) and (\ref{eq conv E.E.R. 2}) prove the proposition.
\end{proof}

%\medskip

%\begin{remark} \label{remark quasi-invariant} It is straightforward to check that the class of rotationally quasi-invariant measures (cf.\ Definition \ref{def rotationally quasi-invariant}) is closed under convex combination and comparability, as in Propositions \ref{prop closure 1} and \ref{prop closure 2}.  However, rotational quasi-invariance is not as well-behaved under product and convolution (since $O(n)\times O(m)$ is a small subgroup of $O(n+m)$).  In the application (Theorem \ref{thm sLSI <=> sHC}), it would suffice to have a weaker quasi-invariance property defined inductively in terms of products; we leave the details out in the absence of a concrete application of such results.
%\end{remark}

\subsection{Continuity of the Dilated Convolution} \label{sect continuity dilated convolution}

\noindent One easy consequence of Definition \ref{def Euclidean regular} is that the operation $f\mapsto f_r$ is bounded on $L^p$.

\begin{lemma} \label{lemma contraction bounded} Let $\mu$ be a Euclidean regular probability measure, let $p>0$, and let $r\in(0,1)$.  Then
\[ \|f_r\|_{L^p(\mu)} \le r^{-n/p}\,C_\mu\left(\textstyle{\frac1r},0\right)^{1/p}\,\|f\|_{L^p(\mu)}. \]
\end{lemma}

\begin{proof} We simply change variables $u=rx$ and use Definition \ref{def Euclidean regular}:
\[ \int |f_r(x)|^p\mu(dx) = \int |f(rx)|^p \rho(x)\,dx = r^{-n}\int |f(u)|^p \rho(x/r)\,dx \le r^{-n} C_\mu\left(\textstyle{\frac1r},0\right)\,\int |f(u)|^p\rho(x)\,dx. \]
\end{proof}

\begin{remark} By condition (\ref{E.R.2}) of Definition \ref{def Euclidean regular}, the constant in Lemma \ref{lemma contraction bounded} is uniformly bounded for $r\in(\e,1]$ for any $\e>0$; that is, there is a uniform (independent of $r$) constant $C_\e$ so that, for $r\in(\e,1]$, $\|f_r\|_{L^p(\mu)} \le C_\e \|f\|_{L^p(\mu)}$. \end{remark}

\medskip

The next proposition shows that, under the assumptions of Definition \ref{def Euclidean regular}, the dilated convolution operation is indeed bounded on $L^p$.
As usual, the conjugate exponent $p'$ to $p\in[1,\infty)$ is defined by $\frac{1}{p}+\frac{1}{p'}=1$. 

\begin{proposition} \label{prop Lp bound} Let $\mu$ be a Euclidean regular probability measure on $\R^n$.  Let $p\in[1,\infty)$, and let $\varphi\in C_c^\infty$ be a test function. Then the {dilated convolution} operation $f\mapsto (f\ast\varphi)_r$ is bounded on $L^p(\mu)$ for each $r\in(0,1)$.  Precisely, if $K = \supp\varphi$ and $s = \sup\{|w|\,;\,w\in K\}$, then
\[ \|(f\ast \varphi)_r\|_{L^p(\mu)} \le r^{-n/p} C_\mu(\textstyle{\frac{1}{r}},\frac{s}{r})^{1/p}\,\mathrm{Vol}(K)^{1/p}\|\varphi\|_{L^{p'}(K)}\; \|f\|_{L^p(\mu)}, \]
where $C_\mu$ is the constant defined in (\ref{eq E.R. constant}).
\end{proposition}
\begin{proof} 
%Since $\rho$ is strictly positive on $\R^n$ and is bounded, 
%It follows that $f$ is in $L^p$ of any compact set.
Denote by $K$ the support of $\varphi$.  By definition,
\[ \|(f\ast\varphi)_r\|_{L^p(\mu)}^p = \int_{\R^n} \left|\int_K f(rx-y)\varphi(y)\,dy\right|^p \,\rho(x)\,dx.\]
We immediately estimate the internal integral using H\"older's inequality:
\[ \left|\int_K f(rx-y)\varphi(y)\,dy\right|^p \le \int_K |f(rx-y)|^p\,dy\cdot \|\varphi\|_{L^{p'}(K)}^p, \]
which is finite since the first integral is the $p$-th power of the  $L^p$-norm of $f$ restricted to the compact set $rx-K$.  Hence,
\begin{equation} \label{eq Lp bound 1} \|(f\ast\varphi)_r\|_{L^p(\mu)}^p \le \|\varphi\|_{L^{p'}(K)}^p \int_{\R^n} \int_K |f(rx-y)|^p dy\, \rho(x)\,dx. \end{equation}
We apply Fubini's theorem to the double integral, which is therefore equal to
\begin{equation}\label{eq Lp bound 2} \int_K \int_{\R^n} |f(rx-y)|^p \rho(x)\,dx\,dy
= \int_K r^{-n}\int_{\R^n} |f(u)|^p \rho\left(\frac{u+y}{r}\right)\,du\,dy \end{equation}
where we have made the change of variables $u=rx-y$ in the internal integral.  
By assumption, $\rho$ is Euclidean regular, and so we have
\begin{equation}\label{eq Lp bound 3}
\rho(\textstyle{\frac{1}{r}}u + \textstyle{\frac{1}{r}}y) \le C_\mu(\textstyle{\frac{1}{r}},\frac{s}{r})\,\rho(u), \quad y\in K. \end{equation}
where $s = \sup\{|w|\,;\,w\in K\}$.  Substituting (\ref{eq Lp bound 3}) into (\ref{eq Lp bound 2}), we see that (\ref{eq Lp bound 1}) yields
\[ \|(f\ast\varphi)_r\|_{L^p(\mu)}^p \le r^{-n}\,C_\mu(\textstyle{\frac{1}{r}},\frac{s}{r})\, \mathrm{Vol}(K)\, \|\varphi\|_{L^{p'}(K)}^p \, \displaystyle{\int |f(u)|^p\rho(u)\,du}. \]
This completes the proof.  \end{proof}

\begin{remark} \label{remark pp' invariant} The explicit constant in Proposition \ref{prop Lp bound} appears to depend strongly on the support set of $\varphi$, but it does not.  Indeed, it is easy to check that the standard rescaling of a test function, $\varphi^s(x) = s^{-n}\varphi(x/s)$, which preserves total mass, also preserves the $\varphi$-dependent quantity above; to be precise, $\mathrm{Vol}(\supp\varphi^s)\|\varphi^s\|^p_{L^{p'}(\R^n)}$ does not vary with $s$.  In addition, the constant $C_\mu(1/r,s/r)$ is well-behaved as $s$ shrinks (indeed, it only decreases).  It is for this reason that the proposition allows us to use the dilated convolution operation with an approximate identity sequence in what follows. \end{remark}

The use of Proposition \ref{prop Lp bound} is that it allows us to approximate an $L^p$ function by smoother $L^p$ functions, along a path through $\LSH$ functions.  To prove this, we first require the following continuity lemma.

\begin{lemma} \label{continuity lemma} Let $\mu$ be a Euclidean regular probability measure, and let $r\in(0,1)$.  Then for any $f\in L^p(\mu)$, the map $T_f\colon\R^n\to L^p(\mu)$ given by $[T_f(y)](x) = f_r(x-y)$ is continuous.
\end{lemma}

\begin{proof} First note that, by the change of variables $u=rx-ry$,
\[ \|T_f(y)\|_{L^p(\mu)}^p = \int |f(rx-ry)|^p\rho(x)\,dx = r^{-n} \int |f(u)|^p \rho\left(\textstyle{\frac{1}{r}}u+y\right)\,du, \]
and the latter is bounded above by $r^{-n} C_\mu(\frac{1}{r},|y|)\, \|f\|_{L^p(\mu)}^p$, showing that the range of $T_f$ is truly in $L^p(\mu)$ for $y\in\R^n$.  Now, fix $\e>0$ and let $\psi\in C_c(\R^n)$ be such that $\|f-\psi\|_{L^p(\mu)} < \e$.  Let $(y_k)_{k=1}^\infty$ be a sequence in $\R^n$ with limit $y_0$.  Then
\[ \|T_f(y_k)-T_f(y_0)\|_{L^p(\mu)} \le
\|T_f(y_k) - T_\psi(y_k)\|_{L^p(\mu)} + \|T_\psi(y_k)-T_\psi(y_0)\|_{L^p(\mu)}
+ \|T_\psi(y_0)-T_f(y_0)\|_{L^p(\mu)}. \]
The first and last terms are simply $T_{\psi-f}(y_k)$ (with $k=0$ for the last term), and so we have just proved that
\[ \|T_{\psi-f}(y_k)\|_{L^p(\mu)} \le r^{-n/p} C_\mu\left(\textstyle{\frac{1}{r}},|y_k|\right)^{1/p} \|\psi-f\|_{L^p(\mu)} < r^{-n/p} C_\mu\left(\textstyle{\frac{1}{r}},|y_k|\right)^{1/p}\e. \]
Moreover, there is a constant $s$ so that $|y_k|\le s$ for all $k$, and since $C_\mu(a,s)$ is an increasing function of $s$, it follows that
\[ \|T_f(y_k)-T_f(y_0)\|_{L^p(\mu)} \le   \|T_\psi(y_k)-T_\psi(y_0)\|_{L^p(\mu)} + 2r^{-n/p} C_\mu\left(\textstyle{\frac{1}{r}},s\right)^{1/p}\e. \]
For each $x$, $(T_\psi(y_k)(x) - T_\psi(y_0)(x) = \psi(rx-ry_k)-\psi(rx-ry_0)$ converges to $0$ since $ry_k\to ry_0$ and $\psi$ is continuous.  In addition, $\psi_r$ is compactly supported and continuous, so it is uniformly bounded.  Since $\mu$ is a probability measure, it now follows that $\|T_\psi(y_k)-T_\psi(y_0)\|_{L^p(\mu)}\to 0$ as $y_k\to y_0$, and the lemma follows by letting $\e\downarrow 0$.
\end{proof}

\begin{corollary} \label{cor convolution} Let $\mu$ be a Euclidean regular probability measure, and let $r\in(0,1)$.  Then for any $f\in L^p(\mu)$, and $\varphi_k$ an approximate identity sequence ($\varphi_k\in C_c^\infty(\R^n)$ with $\int \varphi_k(x)\,dx = 1$ and $\supp\varphi_k \downarrow\{0\}$),
\[ \|f_r\ast\varphi_k-f_r\|_{L^p(\mu)}\to 0 \quad \text{as} \quad k\to\infty. \]
\end{corollary}

\begin{proof} Fix $\e>0$ and let $\psi\in C_c(\R^n)$ be such that $\|f-\psi\|_{L^p(\mu)}<\e$.  The standard $3$ term inequality in this case is
\begin{equation} \label{eq 3 term 2} \|f_r\ast\varphi_k-f_r\|_{L^p(\mu)} \le \|(f_r-\psi_r)\ast\varphi_k\|_{L^p(\mu)} + \|\psi_r\ast\varphi_k-\psi_r\|_{L^p(\mu)} + \|\psi_r - f_r\|_{L^p(\mu)}. \end{equation}
 Following the proof of Lemma \ref{continuity lemma}, we have $\|f_r-\psi_r\|_{L^p(\mu)} \le r^{-n/p}C_\mu(1/r,0)^{1/p}\e$, and from condition (\ref{E.R.2}) of Definition \ref{def Euclidean regular} this is a uniformly bounded constant times $\e$ for $r$ away from $0$.  Also, note that
\[ f_r\ast\varphi_k(x) = \int f_r(x-y)\varphi_k(y)\,dy = \int f(rx-ry)\varphi_k(y)\,dy = r^{-n}\int f(rx-u)\varphi_k(u/r)\,du; \]
that is to say, $f_r\ast\varphi_k = r^{-n}(f\ast\tilde\varphi_k)_r$, where we set $\tilde\varphi_k=(\varphi_k)_{1/r}$.  Hence,
\[ \begin{aligned} 
\|(f-\psi)_r\ast\tilde\varphi_k\|_{L^p(\mu)} &= r^{-n} \|((f-\psi)\ast\tilde\varphi_k)_r\|_{L^p(\mu)}  \\
&\le r^{-n} r^{-n/p} C_\mu\left(\textstyle{\frac{1}{r}},\frac{s_k}{r}\right)^{1/p}\mathrm{Vol}(\supp\tilde\varphi_k)^{1/p}\|\tilde\varphi_k\|_{L^{p'}(\R^n)}\cdot \|f-\psi\|_{L^p(\mu)}  \end{aligned} \]
by Proposition \ref{prop Lp bound}, where $s_k = \sup\{|w|\,;\,w\in\supp\varphi_k\}$.  Since $C_\mu(1/r,\frac{s}{r})$ is increasing in $s$, this constant is uniformly bounded as $k\to\infty$.  What's more, cf.\ Remark \ref{remark pp' invariant}, the product $\mathrm{Vol}(\supp\tilde\varphi_k)^{1/p}\|\tilde\varphi_k\|_{L^{p'}(\R^n)}$ can also be made constant with $k$ (for example by choosing $\varphi_k(x) = k^n \varphi(kx)$ for some fixed unit mass $C_c^\infty$ test-function $\varphi$).  The result is that both the first and last terms in (\ref{eq 3 term 2}) are uniformly small as $k\to\infty$.  Thus, we need only show that $\psi_r\ast\varphi_k\to\psi_r$ in $L^p(\mu)$. The quantity in question is the $p$th root of
\begin{equation} \label{eq convol p norm 1} \int \left|\int\psi_r(x-y)\varphi_k(y)\,dy - \psi_r(x)\right|^p\,\mu(dx)
= \int \left|\int_{K_k}[\psi_r(x-y)-\psi_r(x)]\varphi_k(y)\,dy\right|^p\,\mu(dx), \end{equation}
where we have used the fact that $\varphi_k$ is a probability density; here $K_k$ denotes the support of $\varphi_k$.  Since $\psi_r$ is bounded, we may make the blunt estimate that the quantity in (\ref{eq convol p norm 1}) is
\[ \le \int  \sup_{y\in K_k} |\psi_r(x-y)-\psi_r(x)|^p \left|\int_{K_k} \varphi_k(y)\,dy\right|^p\,\mu(dx)
= \int  \sup_{y\in K_k} |\psi_r(x-y)-\psi_r(x)|^p\,\mu(dx). \]
Since $\psi_r$ is continuous and $K_k$ is compact, there is a point $y_k\in K_k$ such that the supremum is achieved at $y_k$: $\sup_{y\in K_k} |\psi_r(x-y)-\psi_r(x)|^p = |\psi_r(x-y_k)-\psi_r(x)|^p$.  As $k\to\infty$, the support $K_k$ of $\psi_k$ shrinks to $\{0\}$, and so $y_k\to 0$.   The function $|\psi_r(x-y_k)-\psi_r(x)|^p$ is continuous in $x$, and so converges to $0$ pointwise as $y_k\to 0$.  It therefore follows from the dominated convergence theorem that $\|\psi_r\ast\varphi_k-\psi_r\|_{L^p(\mu)}\to 0$, completing the proof.
\end{proof}

We will now use Proposition \ref{prop Lp bound} and Corollary \ref{cor convolution} to prove our main approximation theorem: that $L^p_E(\mu)$ is dense in $L^p(\mu)$ {\em through} log-subharmonic functions.

%However, we now require a stricter growth condition  on the density $\rho$.

%\begin{definition} \label{def exp type} Let $\mu$ be a probability measure on $\R^n$ with density $\rho$.  Let $p\ge 0$.  Say that $\mu$ (or $\rho$) is {\bf exponential type} $p$ if $\rho(x)>0$ for all $x$ and if
%\begin{equation} \label{eq exp type}
%\sup_x \sup_{|y|\le s}|x|^p \frac{\rho(ax+y)}{\rho(x)}<\infty \quad \text{for} \quad a>1, s\ge 0.
%\end{equation}
%\end{definition}

%\begin{remark} Note that a measure of exponential type $p$ is exponential type $q$ for any $q<p$.  It is easy to check that the examples $\rho(x) \sim e^{-c|x|^\alpha}$ from Example \ref{examples E.R.} are exponential type $p$ for every $p>0$; measures of the form $\rho(x) \sim \frac{1}{(1+|x|^2)^\alpha}$ (for $\alpha>\frac{1}{2}$) are exponential type $0$ (i.e. satisfy condition (\ref{E.R.1}) in the definition of Euclidean regular), but are {\em not} exponential type $p$ for any $p>0$.  \end{remark}

\subsection{The Proof of Theorem \ref{thm Sobolev density}} \label{sect proof Sobolev density}

\begin{proof}[Proof of Theorem \ref{thm Sobolev density}] The basic idea of the proof is as follows: approximate a function $f\in \LSH\cap L^p(\mu)$ by $(f\ast\varphi)_r$, and let $\varphi$ run through an approximate identity sequence and $r$ tend to $1$.  We show that the dilated convolution $(f\ast\varphi)_r$ is in $\LSH\cap L^p_E(\mu)$, and that these may be used to approximate $f$ in $L^p$-sense.

\medskip

\noindent {\em Part 1: $(f\ast\varphi)_r$ is in $\LSH\cap L^p_E(\mu)$.}  Let $\varphi\in C_c^\infty(\R^n)$ be a non-negative test function.  Lemma \ref{lemma convolution preserves LSH} shows that $f\ast\varphi$ is $C^\infty$ and $\LSH$.  It is elementary to verify that the cone $C^\infty\cap \LSH$ is invariant under dilations $g\mapsto g_r$; hence the dilated convolution $(f\ast\varphi)_r$ is $C^\infty$ and $\LSH$.  For fixed $r<1$, Proposition \ref{prop Lp bound} shows that $(f\ast\varphi)_r$ is in $L^p(\mu)$, since $f\in L^p(\mu)$.  We must now apply the differential operator $E$.  Note that $(f\ast\varphi)_r$ is $C^\infty$, and so
\[ E[(f\ast\varphi)_r](x) = x\cdot\nabla [(f\ast\varphi)_r](x) = \int rx\cdot\nabla\varphi\,(rx-y)f(y)\,dy. \]
Decomposing $rx=(rx-y)+y$, we break this up as two terms
\begin{equation} \label{eq two terms} E[(f\ast\varphi)_r](x) = \int (rx-y)\cdot\nabla\varphi\,(rx-y)f(y)\,dy + \int y\cdot\nabla\varphi(rx-y)f(y)\,dy. \end{equation}
The first term is just $(f\ast E\varphi)_r(x)$, and since $E\varphi$ is also $C_c^\infty(\R^n)$, Proposition \ref{prop Lp bound} bounds the $L^p$-norm of this term by the $L^p$-norm of $f$.  Hence, it suffices to show that the second term in (\ref{eq two terms}) defines an $L^p(\mu)$-function of $x$.  We now proceed analogously to the proof of Proposition \ref{prop Lp bound}.  Changing variables $u=rx-y$ for fixed $x$ in the internal integral and then using H\"older's inequality,

\[ \begin{aligned} &\int_{\R^n} \left| \int_{\R^n} y\cdot \nabla\varphi\,(rx-y)f(y)\,dy\right|^p \rho(x)\,dx \\
=  &\int_{\R^n} \left| \int_{K} (rx-u)\cdot \nabla\,\varphi(u)f(rx-u)\,du\right|^p \rho(x)\,dx \\
\le &\int_{\R^n} \left(\int_K |rx-u|^p\,|f(rx-u)|^p\,du\right) \, \left(\int_K |\nabla\varphi\,(u)|^{p'}\,dy\right)^{p/p'}\!\rho(x)\,dx,
\end{aligned} \]
where $K=\supp\varphi$.  Note that $\|\nabla\varphi\|_{p'}<\infty$ is a constant independent of $f$.  So we must consider the double integral, to which we apply Fubini's theorem,
\[ \int_{\R^n} \left(\int_K |rx-u|^p |f(rx-u)|^p du\right)\rho(x)\,dx
= \int_K\left(\int_{\R^n} |rx-u|^p |f(rx-u)|^p \rho(x)\,dx\right)du. \]
Now we change variables $v=rx-u$ for fixed $u$ in the internal integral, to achieve
\begin{equation} \label{eq LpE 1} \int_K \left(\int_{\R^n} |v|^p |f(v)|^p \rho\left(\frac{v+u}{r}\right)\,r^{-n}\,dv\right)\,du. \end{equation}
Finally, we utilize the assumption that $\rho$ is exponential type $p$, and so there is a constant $C(p,r,K)$ so that $|v|^p\rho(\frac{v+u}{r}) \le C(p,r,K)\rho(u)$ for $u\in K$.  Hence the integral in (\ref{eq LpE 1}) is bounded above by $C(p,r,K)r^{-n}\mathrm{Vol}(K)$ times the finite norm $\int |f|^p\,d\mu$, which demonstrates that $E[(f\ast\varphi)_r]$ is in $L^p(\mu)$.

\medskip

\noindent {\em Part 2: $(f\ast\varphi)_r$ approximates $f$ in $L^p(\mu)$.}  Let $\varphi_k$ be an approximate identity sequence.  Note by simple change of variables that $(f\ast\varphi_k)_r = f_r \ast (r^n\varphi_k)_r$, and that $(r^n\varphi_k)_r$ is also an approximate identity sequence.  Since $f_r\in L^p(\mu)$, by Lemma \ref{lemma contraction bounded}, it follows from Corollary \ref{cor convolution} that $(f\ast\varphi_k)_r\to f_r$, $k\rightarrow\infty$, in $L^p(\mu)$.  We must now show that $f_r\to f$ in $L^p(\mu)$ as $r\uparrow 1$.  For this purpose, once again fix $\e>0$ and choose a $\psi\in C_c(\R^n)$ so that $\|f-\psi\|_{L^p(\mu)}<\e$.  Then
\begin{equation} \label{eq last 3 term} \|f-f_r\|_{L^p(\mu)} \le \|f-\psi\|_{L^p(\mu)} +\|\psi-\psi_r\|_{L^p(\mu)} + \|\psi_r-f_r\|_{L^p(\mu)}. \end{equation}
The first term is $<\e$, and changing variables the last term is
\[ \begin{aligned} \|\psi_r-f_r\|_{L^p(\mu)}^p = \int |\psi(rx)-f(rx)|^p\rho(x)\,dx &= r^{-n}\int |\psi(u)-f(u)|^p \rho(u/r)\,du \\
&\le r^{-n}C_\mu\left(\textstyle{\frac1r},0\right)\int|\psi-f|^p\,d\mu.
\end{aligned}  \]
Here we have used the fact that $\mu$ is Euclidean regular.  Note that, by condition (\ref{E.R.2}) of Definition \ref{def Euclidean regular}, the constant appearing here is uniformly bounded by, say, $C$, for $r\in(\frac12,1]$. Thence, the last term in (\ref{eq last 3 term}) is bounded above by $C^{1/p}\e$ and is also uniformly small.  Finally, the middle term tends to $0$ as $r\uparrow 1$ since $\psi_r\to\psi$ pointwise and the integrand is uniformly bounded.  Letting $\e$ tend to $0$ completes the proof.
\end{proof}

\section{The Intrinsic Equivalence of (\ref{sLSI}) and (\ref{sHC})} \label{sect SHC ==> LSI}

In this section, we prove Theorem \ref{thm sLSI <=> sHC}: if a measure $\mu$  is sufficiently Euclidean regular (satisfying the conditions of Definition \ref{def Euclidean regular}), and if {\em $\mu$ is invariant under rotations}, then $\mu$ satisfies a strong log-Sobolev inequality precisely
%%%%%%%%%%%%%%%%%%%%%%%%%%%%%%%%
when
%%%%%%%%%%%%%%%%%%%%%%%%%%%%%%%
it satisfies strong hypercontractivity.  It will be useful to fix the following notation.

\begin{notation} \label{notation q} Let $c>0$ be a fixed constant, let $\mu$ be a measure on $\R^n$, and let $f$ be a function on $\R^n$.
\begin{enumerate}
\item For $r\in(0,1]$, let $q = q(r)$ denote the function
\[ q(r) = r^{-2/c}. \]
Note that $q\in C^\infty(0,1]$, is decreasing, and $q(1)=1$.

\item Define a function $\alpha_{f,\mu}\colon (0,1]\to[0,\infty)$ by
\[ \alpha_{f,\mu}(r) \equiv \|f_r\|_{L^{q(r)}(\mu)} = \left(\int |f(rx)|^{q(r)}\,\mu(dx)\right)^{1/q(r)}. \]
When the function $f$ and measure $\mu$ are clear from context, we denote $\alpha_{f,\mu} = \alpha$.
\end{enumerate} \end{notation}

\noindent We begin with the following general statement.

\begin{lemma} \label{lemma derivative} Suppose $\mu$ is a Euclidean regular probability measure.  Let $q_0>1$, and let $f\ge0$ be in $L^{q_0}(\mu)\cap C^\infty(\R^n)$. Let $\e\in(0,1)$, and suppose there are functions $h_1,h_2\in L^1(\mu)$ such that for all $r\in(\e,1]$,
\begin{equation} \label{eq terms to bound} |f(rx)^{q(r)}\log f(rx)| \le h_1(x), \quad |f(rx)^{q(r)-1}Ef(rx)| \le h_2(x) \quad a.s.[x]. \end{equation}
Then there is $\e'\in(\e,1)$ such that $\alpha = \alpha_{f,\mu}$ is differentiable on $(\e',1]$, and for $r$ in this domain,
\begin{equation} \label{eq derivative 0} \begin{aligned}
\alpha'(r) = \frac{2}{crq(r)}\|f_r\|_{q(r)}^{1-q(r)}\Bigg[ \|f_r\|_{q(r)}^{q(r)} \log \|f_r\|_{q(r)}^{q(r)} &- \int f(rx)^{q(r)} \log f(rx)^{q(r)}\,\mu(dx) \\
&+ \frac{cq(r)}{2}\int f(rx)^{q(r)-1} Ef(rx)\,\mu(dx)\Bigg].
\end{aligned} \end{equation}
\end{lemma}

\begin{remark} Note that $(1/q(r))^{c/2} = r$.   Hence, if $f\in\LSH$ and $\mu$ satisfies the strong hypercontractivity property of (\ref{sHC}) (with $p=1$) we have $\alpha(r) \le \|f\|_1 = \alpha(1)$ for $r\in(0,1]$.  The conditions of Lemma \ref{lemma derivative} guarantee that $\alpha$ is differentiable; hence, we essentially have that $\alpha'(1)\ge 0$.  Equation (\ref{eq derivative 0}) shows that $\alpha'(1)$ is closely related to the expression in (\ref{sLSI}), and indeed this is our method for proving the logarithmic Sobolev inequality in what follows.
\end{remark}

\begin{proof} Set $\beta(r,x) = f(rx)^{q(r)}$, so that $\alpha(r)^{q(r)} = \int \beta(r,x)\,\mu(dx)$.  Note, $\beta(r,x) = f_r(x)^{q(r)}$.  The function $q(r)$ is
%decreasing and
continuous and $q(1)=1$, so there is $\e'>0$ so that $q(r)<q_0$ for $r\in(\e',1)$; and hence $f^{q(r)}\in L^1(\mu)$.  (We increase $\e'$ if necessary so $0<\e<\e'$.)  As $\mu$ is Euclidean regular, Lemma \ref{lemma contraction bounded} shows that $f_r^{q(r)}$ is also in $L^1(\mu)$, and so $\beta(r,\cdot)\in L^1(\mu)$ for all $r\in(\e',1)$.  Since $f\in C^\infty$, we can check quickly that $\beta(\cdot,x)$ is as well; using the fact that $q'(r) = -\frac2c r^{-2/c-1} = -\frac{2}{cr}q(r)$, and that $\frac{\del}{\del r} f(rx) = \frac{1}{r}Ef(rx)$, logarithmic differentiation yields
\begin{equation} \label{eq derivative 1}
\frac{\del}{\del r} \beta(r,x) = q(r)\left[-\frac{2}{cr} f(rx)^{q(r)}\log f(rx) + \frac{1}{r}f(rx)^{q(r)-1} Ef(rx)\right]. \end{equation}
From the hypotheses of the Lemma, we therefore have
\[ \left|\frac{\del}{\del r}\beta(r,x)\right| \le \frac{q(r)}{r}\left[\frac{2}{c} h_1(x) + h_2(x)\right] \]
for almost every $x\in\R^n$, for $r\in(\e',1]$.  As $q(r)/r$ is uniformly bounded on $(\e',1]$, we see that $|\frac{\del}{\del r}\beta(r,x)|$ is uniformly bounded above by an $L^1(\mu)$ function.  It now follows from the Lebesgue differentiation theorem that $\alpha(r)^{q(r)} = \int \beta(r,x)\,\mu(dx)$ is differentiable on a neighbourhood of $1$, and
\begin{equation} \label{eq derivative 2} \begin{aligned}
\frac{d}{dr}\left[\alpha(r)^{q(r)}\right] &= \int \frac{\del}{\del r}\beta(r,x)\,\mu(dx) \\
&= -\frac{2}{cr}q(r) \int f(rx)^{q(r)}\log f(rx)\,\mu(dx) + \frac{1}{r}q(r) \int f(rx)^{q(r)-1} Ef(rx)\,\mu(dx). \end{aligned}
\end{equation}
Consequently $\alpha(r)$ is differentiable in a neighbourhood of 1. Again using logarithmic differentiation,
\[ \alpha'(r) = \alpha(r)\frac{d}{dr}\log\alpha(r) = \alpha(r)\frac{d}{dr}\left[ \frac{1}{q(r)} \log \alpha(r)^{q(r)}\right], \]
and again using the fact that $q'(r) = -\frac{2}{cr}q(r)$,
\[ \begin{aligned} \frac{d}{dr}\left[ \frac{1}{q(r)} \log \alpha(r)^{q(r)}\right]
&= \frac{2}{crq(r)} \log \alpha(r)^{q(r)} + \frac{1}{q(r)} \alpha(r)^{-q(r)}\frac{d}{dr}\left[\alpha(r)^{q(r)}\right] \\
&= \frac{\alpha(r)^{-q(r)}}{q(r)}\left(\frac{2}{cr}\alpha(r)^{q(r)} \log\alpha(r)^{q(r)} + \frac{d}{dr}\left[\alpha(r)^{q(r)}\right]\right).
\end{aligned} \]
%As $\alpha(r)$ is bounded away from $0$ (we exclude the case $f=0$ of course), it now follows that $\alpha$ is differentiable in a neighbourhood of $1$. 
Combining with (\ref{eq derivative 2}), we therefore have
\begin{equation} \label{eq derivative 3} \begin{aligned}
\alpha'(r) = \frac{\alpha(r)^{1-q(r)}}{q(r)} \Bigg[\frac{2}{cr}\alpha(r)^{q(r)}\log\alpha(r)^{q(r)} &- \frac{2}{cr}q(r) \int f(rx)^{q(r)}\log f(rx)\,\mu(dx) \\
&+ \frac{1}{r}q(r) \int f(rx)^{q(r)-1} Ef(rx)\,\mu(dx) \Bigg]
\end{aligned} \end{equation}
Simplifying (\ref{eq derivative 3}), and using the definition $\alpha(r) = \|f_r\|_{q(r)}$, yields (\ref{eq derivative 0}), proving the lemma.
\end{proof}

We therefore seek conditions on a function $f$ (and on the measure $\mu$) which guarantee the hypotheses of Lemma \ref{lemma derivative} (specifically the existence of the Lebesgue dominating functions $h_1$ and $h_2$).  Naturally, we will work with $\LSH$ functions $f$.  We will also make the fairly strong assumption that $\mu$ is rotationally-invariant.

\begin{notation} \label{def average} Let $f\colon\R^n\to\R$ be locally-bounded.  Denote by $\tilde{f}$ the spherical average of $f$.  That is, with $\vartheta$ denoting Haar measure on the group $O(n)$ of rotations of $\R^n$,
\[ \tilde{f}(x) = \int_{O(n)} f(ux)\,\vartheta(du). \]
\end{notation}

\noindent If $\mu$ is rotationally-invariant, then $\int f\,d\mu = \int \tilde{f}\,d\mu$ for any $f\in L^1(\mu)$.  As such, we can immediately weaken the integrability conditions of Lemma \ref{lemma derivative} as follows.

\begin{lemma} \label{lemma derivative rotation} Suppose $\mu$ is a Euclidean regular probability measure that is invariant under rotations of $\R^n$.  Let $q_0>1$ and let $f\ge 0$ be in $L^{q_0}(\mu)\cap C^\infty(\R^n)$.  Denote by $f_1,f_2\colon(0,1]\times\R^n\to\R$ the functions
\begin{equation} \label{eq f1 f2} f_1(r,x) = f(rx)^{q(r)}\log f(rx), \quad f_2(r,x) = f(rx)^{q(r)-1} Ef(rx). \end{equation}
Fix $\e\in(0,1)$, and suppose that there exist functions $h_1,h_2\in L^1(\mu)$ such that, for $r\in(\e,1]$, $|\tilde{f_j}(r,x)|\le h_j(x)$ for almost every $x$, $j=1,2$.  (Here $\tilde{f}_j(r,\cdot)$ refers to the rotational average of $f_j(r,\cdot)$, as per Notation \ref{def average}.)  Then the conclusion of Lemma \ref{lemma derivative} stands: for some $\e'\in(\e,1)$, the function $\alpha=\alpha_{f,\mu}$ is differentiable on $(\e',1]$, and its derivative is given by (\ref{eq derivative 0}).  \end{lemma}

\begin{proof} Following the proof of Lemma \ref{lemma derivative}, only a few modifications are required.  Defining $\beta(r,x)$ as above, $\alpha(r)^{q(r)} = \int \beta(r,x)\,\mu(dx)$; since $\mu$ is rotationally-invariant, this is equal to $\int \tilde{\beta}(r,x)\,\mu(dx)$ where $\tilde{\beta}$ refers to the rotational average of $\beta$ in the variable $x$.  Evidently $\tilde{\beta}(r,\cdot)$ is $\mu$-integrable for sufficiently large $r<1$ (since $\beta$ is).  To use the Lebesgue differentiation technique, we must verify that $\frac{\del}{\del r}\tilde{\beta}(r,x)$ exists for almost every $x$ and is uniformly bounded by an $L^1(\mu)$ dominator.  Note that $\beta(r,x)$ is locally-bounded in $x$ for each $r$, and so for fixed $x$ it is easy to verify that indeed
\[ \frac{\del}{\del r} \tilde{\beta}(r,x) = \int_{O(n)} \frac{\del}{\del r} \beta(r,ux)\,\vartheta(du). \]
Using (\ref{eq derivative 1}), we then have
\[ \frac{\del}{\del r}\tilde{\beta}(r,x) = q(r)\int_{O(n)}\left(-\frac{2}{cr} f(rux)^{q(r)}\log f(rux) + \frac{1}{r}f(rux)^{q(r)-1}Ef(rux)\right)\,\vartheta(du). \]
That is, using (\ref{eq f1 f2}), $\frac{\del}{\del r}\tilde{\beta}(r,x) = q(r)\left[-\frac{2}{cr} \tilde{f_1}(r,x) +\frac{1}{r}\tilde{f_2}(r,x)\right]$.   Hence, from the assumptions of this lemma,
\[ \left|\frac{\del}{\del r}\tilde{\beta}(r,x)\right| \le \frac{q(r)}{r}\left[\frac{2}{c} h_1(x) + h_2(x)\right] \]
and so, since $q(r)/r$ is uniformly bounded for $r\in(\frac12,1]$, it follows that $\alpha(r)^{q(r)} = \int \tilde{\beta}(r,x)\,\mu(dx)$ is differentiable near $1$, with derivative given by
\[ \int \frac{\del}{\del r}\tilde{\beta}(r,x)\,\mu(dx) = q(r)\left[ -\frac{2}{rc}\int \tilde{f_1}(r,x)\,\mu(dx) + \frac{1}{r}\int \tilde{f_2}(r,x)\,\mu(dx).\right] \]
Now using the rotational-invariance of $\mu$ again, these integrals are the same as the corresponding non-rotated integrands $\int f_j(r,x)\,\mu(dx)$, yielding the same result as (\ref{eq derivative 2}).  The remainder of the proof follows the proof of Lemma \ref{lemma derivative} identically.
\end{proof}

\begin{remark} The point of Lemma \ref{lemma derivative rotation} -- that it is sufficient to find uniform Lebesgue dominators for the {\em rotational averages} of the terms in (\ref{eq terms to bound}) -- is actually quite powerful for us.  While a generic subharmonic function in dimension $\ge 2$ may not have good global properties, a {\em rotationally-invariant} subharmonic function does, as the next proposition demonstrates.  We will exploit this kind of behaviour to produce the necessary bounds to verify the conditions of Lemma \ref{lemma derivative rotation} and prove the differentiability of the norm. \end{remark}

\begin{proposition} \label{prop sh rotation} Let $f\colon\R^n\to\R$ be subharmonic and locally-bounded.  Then $\tilde{f}$ is also subharmonic; moreover, for fixed $x\in\R^n$, $r\mapsto \tilde{f}(rx)$ is an increasing function of $r\in[0,1]$. \end{proposition}

\begin{proof} %This was proved in \cite{GKL}; we repeat the quick proof here for the reader's convenience. 
Fix $u\in O(n)$.  Since $f$ is locally-bounded, subharmonicity means that $\fint_{B(x,r)} f(t)\,dt \ge f(x)$ for every $x\in\R^n$, $r\in(0,\infty)$.  Changing variables, we have
\[ \fint_{B(x,r)} f(\alpha t)\,dt = \fint_{u\cdot B(x,r)} f(t)\,dt = \fint_{B(u x,r)} f(t)\,dt \ge f(ux). \]
Hence, $f\circ u$ is subharmonic for each $u\in O(n)$.  The local-boundedness of $f$ means that the function $u\mapsto f(ux)$ is uniformly bounded in $L^1(O(n),\vartheta)$ for $x$ in a compact set, and hence it follows that $\tilde{f}$ is subharmonic.

\medskip

\noindent  Hence $\tilde{f}$ is a rotationally-invariant subharmonic function.  Fix $x\in\R^n$ and $r\in[0,1]$.  Then $rx$ is in the ball $B(0,|x|)$, and since $\tilde{f}$ is subharmonic, the maximum principle asserts that $\tilde{f}(rx)$ is no larger than the maximum of $f$ on $\del B(0,|x|)$.  But $\tilde{f}$ is constantly equal to $\tilde{f}(x)$ on $\del B(0,|x|)$ by rotational-invariance, and so $\tilde{f}(rx)\le \tilde{f}(x)$, proving the proposition.
\end{proof}

Proposition \ref{prop sh rotation} makes it quite easy to provide a uniform Lebesgue dominating function for the function $f_1$ in Lemma \ref{lemma derivative rotation}.

\begin{proposition} \label{prop bound f1} Suppose $\mu$ is a rotationally-invariant probability measure on $\R^n$.  Let $q_0>1$, and let $f\ge 0$ be {\em subharmonic} and in $L^{q_0}(\mu)$.  Define $f_1$ as in (\ref{eq f1 f2}): $f_1(r,x) = f(rx)^{q(r)} \log f(rx)$.  Set $g_1(x) = f(x)^{q_0}$, and set $h_1 = \tilde{g_1}+1$; i.e.\ $h_1(x)=1+ \int_{O(n)} f(ux)^{q_0}\,\vartheta(du)$. Then $h_1\in L^1(\mu)$ and  there is an $\e\in(0,1)$ and a constant $C>0$ so that for all $r\in(\e,1]$, $|\tilde{f_1}(r,x)| \le C h_1(x)$ for almost every $x$. \end{proposition}

\begin{remark} By the rotational-invariance of $\mu$, $\int h_1\,d\mu = \int \tilde{g_1}\,d\mu+1 = \int g_1\,d\mu+1 = \int f^{q_0}\,d\mu+1 <\infty$, and so $h_1$ is a uniform $L^1(\mu)$ dominator verifying the first condition of Lemma \ref{lemma derivative rotation}. \end{remark}

\begin{proof} Choose some small $\delta\in(0,1)$.  First note from simple calculus that, for $u\ge 1$, $u^{-\delta}\log u \le \frac{1}{e\delta}$. Now, choose $\e\in(0,1)$ so that $q\e < q_0-\delta$; then $q(r) < q_0-\delta$ for $r\in(\e,1]$.  Consequently, if $f(y)\ge 1$, we have
\[ 0\le f(y)^{q(r)}\log f(y) \le  f(y)^{q_0-\delta}\log f(y) \le \frac{1}{e\delta} f(y)^{q_0}. \]
On the other hand, for $0\le u\le 1$, $|u^{q(r)}\log u| \le \frac{1}{eq(r)} \le \frac1e$ (again by simple calculus).  Thus, since $f\ge 0$, in total we have
\begin{equation} \label{eq log estimate}
|f(y)^{q(r)}\log f(y)| \le \frac{1}{e}\max\left\{\frac{1}{\delta}f(y)^{q_0},1\right\} \le \frac{1}{e\delta}[f(y)^{q_0}+1].
\end{equation}
Set $C = \frac{1}{e\delta}$.  With $y=rx$, the left-hand-side of (\ref{eq log estimate}) is precisely $f_1(r,x)$.  Averaging (\ref{eq log estimate}) over $O(n)$ and recalling that $g_1(y)=f(y)^{q_0}$, we have
\[ |\tilde{f_1}(r,x) \le C [\tilde{g_1}(rx)+1]. \]
Recall that if $\varphi$ is convex and $f$ is subharmonic then $\varphi\circ f$ is also subharmonic.  Thus, since $q_0>1$ and $f$ is subharmonic, $g_1$ is also subharmonic, and hence from Proposition \ref{prop sh rotation}, $\tilde{g_1}(rx)\le \tilde{g_1}(x)$.  This proves the proposition.
\end{proof}

\noindent We must now bound the second term $\tilde{f_2}(r,\cdot)$ uniformly for $r$ in a neighbourhood of $1$.  The following Lemma is useful in this regard.

\begin{lemma} \label{lemma bound Ek} Let $\tilde{k}$ be a $C^\infty$ non-negative subharmonic rotationally-invariant function.  Then for $x\in\R^n$ and $r\in(0,1]$,
\begin{equation} \label{eq bound Ek} E\tilde{k}(rx) \le r^{2-n} E\tilde{k}(x). \end{equation}
\end{lemma}

\begin{proof} Since $\tilde{k}$ is rotationally-invariant, there is a function $h\colon[0,\infty)\to\R$ so that $\tilde{k}(x) = h(|x|)$.  The Laplacian of $\tilde{k}$ can then be expressed in terms of derivatives of $h$; the result is
\begin{equation} \label{eq Laplacian} \Delta\tilde{k}(x) = h''(|x|) + (n-1)\frac{1}{|x|}h'(|x|). \end{equation}
Hence, since $\tilde{k}$ is subharmonic and smooth, it follows that for $t>0$,
\begin{equation} \label{eq Laplacian 2} t\,h''(t) + (n-1)h'(t) \ge 0. \end{equation}
One can also check that, in this case, $E\tilde{k}(x) = |x|h'(|x|)$.  Now, define $F(r) = r^{n-2} E\tilde{k}(rx) = r^{n-2} r|x| h'(r|x|)$.  Then $F$ is smooth on $(0,\infty)$ and $F(1) = |x|h'(|x|) = E\tilde{k}(x)$.  We differentiate, yielding
\[ \begin{aligned} F'(r) &= |x|\frac{d}{dr} r^{n-1}h'(r|x|) = |x|(n-1)r^{n-2} h'(r|x|) + |x|r^{n-1} h''(r|x|)|x| \\
&= |x|r^{n-2}\left[r|x|h''(r|x|) + (n-1)h'(r|x|)\right].
\end{aligned} \]
Equation (\ref{eq Laplacian 2}) with $t=r|x|$ now yields that $F'(r)\ge 0$ for $r>0$.  Hence, $F(r)\le F(1)$ for $r\le 1$.  This is precisely the statement of the lemma.
\end{proof}

\begin{proposition} \label{prop bound f2} Let $q_0>1$ and let $\mu$ be a rotationally-invariant probability measure on $\R^n$.  Let $f\ge 0$ be subharmonic, $C^\infty$, and {\em in $L^{q_0}_E(\mu)$}.  Define $f_2$ as in (\ref{eq f1 f2}): $f_2(r,x) = f(rx)^{q(r)-1} Ef(rx)$.  Set $g_3(x) = (f(x)^{q_0-1}+1)|Ef(x)|$, and set $h_2 = \tilde{g_3}$.  Then there is an $\e\in(0,1)$ and a constant $C>0$ so that for all $r\in(\e,1]$, $|\tilde{f_2}(r,x)|\le Ch_2(x)$ for almost every $x$; moreover, $h_2\in L^1(\mu)$.
\end{proposition}

\begin{proof} Fix $\e\in(0,1)$ small enough that $q(r)<q_0$ for all $r\in(\e,1]$.  Define $g_2(r,y) = f(y)^{q(r)-1}Ef(y)$. and note that $f_2(r,x)$ is given by the dilation $f_2(r,x) = g_2(r,rx)$.  Since $E$ is a first-order differential operator, we can quickly check that
\[ g_2(r,y) =\frac{1}{q(r)}E(f^{q(r)})(y). \]
We now average both sides over $O(n)$.  Set $k=f^{q(r)}$, which is $C^\infty$, and let $u\in O(n)$.  Then we have the following calculus identity:
\[ E(k\circ u)(y) = y\cdot \nabla(k\circ u)(y) = y\cdot u^\top\nabla k(u y) = (u y)\cdot \nabla k (u y) = (Ek)(u y). \]
%where the second equality follows from the chain rule since $\nabla k = (Dk)^\top$ and the third is true because $\alpha\in O(n)$.
For fixed $y$ the function $u\mapsto (Ek)(uy)$ is uniformly bounded and so we integrate both sides to yield
\[ \widetilde{Ek}(y) = \int_{O(n)} (Ek)(uy)\,\vartheta(du) = \int E(k\circ u)(y)\,\vartheta(du) = E\int k\circ u (y)\,\vartheta(du) = E(\tilde{k})(y). \]
In other words, $\tilde{g_2}(r,y) =\frac{1}{q(r)} E(\widetilde{f^{q(r)}})(y)$.  As in the proof of Proposition \ref{prop bound f1}, the function $\tilde{k}=\widetilde{f^{q(r)}}$ is subharmonic, and rotationally invariant.  Hence, we employ Lemma \ref{lemma bound Ek} and have
\[ \tilde{g_2}(r,rx) =\frac{1}{q(r)} E\tilde{k}(rx) \le\frac{1}{q(r)} r^{2-n}E\tilde{k}(x) =r^{2-n} \tilde{g_2}(r,x). \]
Since $r^{2-n}$ is uniformly bounded for $r\in(\e,1]$, it now suffices to find a uniform dominator for $\tilde{g_2}(r,x)$.

\medskip

\noindent We therefore make the estimates: since $q(r)<q_0$ we have
\[\begin{aligned} |g_2(r,x)| =  f(x)^{q(r)-1}|Ef(x)| \le \max\{f(x)^{q(r)-1},1\}|Ef(x)| &\le \max\{f(x)^{q_0-1},1\}|Ef(x)| \\
&\le \left(f(x)^{q_0-1}+1\right)|Ef(x)|. \end{aligned}\]
That is to say, $|g_2(r,x)| \le g_3(x)$ for $r\in(\e,1]$.  Hence,
\[ |\tilde{g_2}(r,x)| = \left|\int_{O(n)} g_2(r,ux)\,\vartheta(du)\right| \le \int_{O(n)} |g_2(r,ux)|\,\vartheta(du) \le \int_{O(n)} g_3(ux)\,\vartheta(du) = \tilde{g_3}(x) = h_2(x),  \]
thus proving the estimate.

\medskip

\noindent As usual, by rotational invariance of $\mu$, $\int \tilde{g_3}\,d\mu = \int g_3\,d\mu$, and so to show $h_2\in L^1(\mu)$ we need only verify that $g_3\in L^1(\mu)$.  To that end, we break up $g_3(x) = f(x)^{q_0-1}|Ef(x)| + |Ef(x)|$.  By assumption, $f\in L^{q_0}_E(\mu)$ and so $|Ef|\in L^{q_0}(\mu)$; as $\mu$ is a finite measure, this means that $|Ef|\in L^1(\mu)$ and hence the second term is integrable.  For the first term, we use H\"older's inequality:
\[ \int f^{q_0-1} |Ef|\,d\mu \le \| f^{q_0-1} \|_{q_0'} \|Ef\|_{q_0} = \|f\|_{q_0}^{q_0-1} \|Ef\|_{q_0}. \]
Both terms are finite since $f\in L^{q_0}_E(\mu)$, and hence $g_3\in L^1(\mu)$, proving the proposition.
\end{proof}

\noindent Combining Lemma \ref{lemma derivative rotation} and Propositions \ref{prop bound f1} and \ref{prop bound f2}, we therefore have the following.

\begin{theorem} \label{theorem ==> LSI}
Let $q_0>1$ and let $\mu$ be a probability measure of Euclidean type $q_0$, that is invariant under rotations of $\R^n$.  Suppose that $\mu$ satisfies strong hypercontractivity of (\ref{sHC}) with constant $c>0$.  Let $f\in L^{q_0}_E(\mu)\cap\LSH\cap C^\infty$.  Then the strong log-Sobolev inequality, (\ref{sLSI}), holds for $f$:
\[ \int f\log f\,d\mu - \int f\,d\mu \,\log\int f\,d\mu \le \frac{c}{2}\int Ef\,d\mu. \]
\end{theorem}

\begin{proof} Under the conditions stated above, the results of the preceding section show that the function $\alpha=\alpha_{f,\mu}$ is differentiable on $(\e',1]$ for some $\e'\in(0,1)$.  Since $\mu$ satisfies strong hypercontractivity, Proposition \ref{prop SHC increasing} shows that the function $\alpha$ is non-decreasing on $(0,1]$.  It therefore follows that $\alpha'(r) \ge 0$ for $r\in(\e',1]$ (here $\alpha'(1)$ denotes the left-derivative).  Hence, from (\ref{eq derivative 0}) we have, for $r\in(\e',1]$,
\[
\|f_r\|_{q(r)}^{q(r)} \log \|f_r\|_{q(r)}^{q(r)} - \int f(rx)^{q(r)} \log f(rx)^{q(r)}\,\mu(dx)
+ \frac{cq(r)}{2}\int f(rx)^{q(r)-1} Ef(rx)\,\mu(dx) \ge 0.
\]
At $r=1$, this reduces precisely to (\ref{sLSI}), proving the result.
\end{proof}

Theorem \ref{theorem ==> LSI} is part (2) of Theorem \ref{thm sLSI <=> sHC}.  The proof of (1) is essentially the same.

%\subsection{Examples}
%
%1. Consider $d=1$. The symmetric measures sur $\R$ considered  in the Example \ref{examples E.R.}, i.e. 
%
% -- measures with  a density $\frac{1}{(1+x^2)^\alpha}$ and $\alpha>1/2$  
% 
% -- measures with  densities  $e^{-a|x|^\alpha}$ with $a>0$\\
% have the property of strong hypercontractivity with $c=2$ (see \cite{GKL}), so by  Theorem \ref{theorem ==> LSI} ,
% the  strong LSI holds for them with $c=2$.\\
% 
% 2. It would be interesting to prove that the  measures with  densities  $e^{-a|x|^\alpha}$
% on $\R^n$ have strong hypercontractivity.\\
% 
% {\bf JEAN-J. HAS A PROOF THAT for a convenient function $h$,
% the densities $e^{-a|x|^\alpha}h(x)$ satisfy
% strong hypercontractivity}.\\
% 
% 3. The measures with densities  $e^{-a|x|^\alpha}$ on $\R$ satisfy strong hypercontractivity
% with constant 2  (see \cite{GKL}. However we know that for $\alpha=2$, the best constant is $c=1$.
% We conjecture that the best constant for  $e^{-a|x|^\alpha}$ is $c=2/\alpha$.
% 
% {\bf JEAN-J. HAS A PROOF THAT for $\alpha=1$ the conjecture is true, i.e. the  smallest $c=2$.}\\
% 
% 4.  It would be interesting to study the strong hypercontractivity and hypercontractivity
% for rotationally invariant $\alpha$-stable measures on $\R^n$, $0<\alpha<2$. Recall that their Fourier tranform equals
% $e^{-a|x|^\alpha}$.
% 
% 5. Prof. Yuichi Kanjin called our attention to the relations
% between hypercontractivity and   Hausdorff-Young inequalities
% of the form 
% $$ \|{\mathcal F}(f)\|_q \le c \|f\|_p $$
% with $c<1$, see   articles by Beckner, Pabenko and Lieb.

\end{document}